\documentclass[11pt, letter]{article}
         
\usepackage{amsfonts,amssymb, latexsym, amsmath, amsthm,mathrsfs, verbatim,geometry}
\usepackage{graphics}
\usepackage{slashed}
\usepackage{url,color}
\usepackage{enumerate}
\usepackage{esint}
\usepackage{pifont}
\usepackage{upgreek}
\usepackage{times}
\usepackage{calligra}
\usepackage{graphicx}
\usepackage{caption}
\usepackage{tikz}
\usepackage{accents}
\usetikzlibrary{calc}
\usetikzlibrary{arrows.meta}

\allowdisplaybreaks

\newcommand\J{\mathscr J}
\newcommand\D{\mathcal D}
\newcommand\R{\mathcal R}
\newcommand\SC{\mathcal S}
\newcommand\PC{\mathcal P}

\newtheorem{theorem}{Theorem}[section]

\newtheorem{proposition}[theorem]{Proposition}

\newtheorem{cor}[theorem]{Corollary}
\newtheorem{lemma}[theorem]{Lemma}
\newtheorem{remark}[theorem]{Remark}
\newtheorem*{proposition*}{Proposition}

\renewcommand{\theequation}{\arabic{section}.\arabic{equation}}

\newtheorem{manualtheoreminner}{Theorem}
\newenvironment{manualtheorem}[1]{%
  \renewcommand\themanualtheoreminner{#1}%
  \manualtheoreminner
}{\endmanualtheoreminner}

\def\bcb{\begin{color}{blue}}
\def\bcr{\begin{color}{red}} 
\def\bcv{\begin{color}{violet}} 
\def\ec{\end{color}}

\title{The vacuum boundary problem for the spherically symmetric compressible Euler equations with positive density and unbounded entropy}

\author{Calum Rickard\footnote{Department of Mathematics, University of Southern California, Los Angeles, USA}}

\date{}
\begin{document}

\maketitle

\abstract{Global stability of the spherically symmetric nonisentropic compressible Euler equations with positive density around global-in-time background affine solutions is shown in the presence of free vacuum boundaries. Vacuum is achieved despite a non-vanishing density by considering a negatively unbounded entropy and we use a novel weighted energy method whereby the exponential of the entropy will act as a changing weight to handle the degeneracy of the vacuum boundary. Spherical symmetry introduces a coordinate singularity near the origin for which we adapt a method developed for the Euler-Poisson system \cite{guo2018continued} to our problem.
}

\section{Introduction}
We consider the free boundary compressible Euler equations for ideal gases in three space dimensions
\begin{alignat}{2} 
\rho(\partial_t \mathbf{u} + \mathbf{u} \cdot \nabla \mathbf{u}) + \nabla p  & = 0 & \qquad \text{on} \enskip \Omega(t), \label{E:MOM} \\
\partial_t \rho + \text{div}(\rho\mathbf{u})   & = 0 &\qquad \text{on} \enskip \Omega(t), \label{E:M} \\ 
\partial_t S + \mathbf{u} \cdot \nabla S & = 0 &\qquad \text{on} \enskip \Omega(t), \label{E:S}
\end{alignat}
coupled with equation of state for an ideal gas
\begin{equation} 
p(\rho,S)=\rho^{\gamma} e^{S}, \label{E:EEOS}
\end{equation}
where $\mathbf{u}$ is the fluid velocity vector field, $\rho$ is the density, $S$ is the entropy, $p$ is the pressure, $\gamma > 1$ is the adiabatic constant and $\Omega(t) \subset \mathbb{R}^3$ is a time dependent, open bounded domain with boundary $\partial \Omega(t)$ where $t \in [0,T]$ for some $T>0.$

Our boundary conditions are the physical vacuum boundary condition coupled with kinematic boundary condition
\begin{alignat}{2}
p&=0 & \qquad \text{on} \enskip \partial \Omega(t), \label{E:VACUUM} \\
 -\infty < \frac{\partial c_s^2}{\partial \mathbf{n}}  &<0 & \qquad \text{on} \enskip \partial \Omega(t), \label{E:ENORMALDERIVATIVE} \\ 
 \mathcal{V} (\partial\Omega(t))&= \mathbf{u}\cdot \mathbf{n}(t) & \qquad \text{on} \enskip \partial \Omega(t),\label{E:VELOCITYBDRYE}       
 \end{alignat}   
with $\mathbf{n}$ the outward unit normal vector to $\partial\Omega(t)$, $\frac{\partial}{\partial \mathbf{n}}$ the outward normal derivative, $c_s:=\sqrt{\frac{\partial p}{\partial \rho}}=\sqrt{\gamma \rho^{\gamma-1} e^S}$ is the speed of sound and $\mathcal{V} (\partial\Omega(t))$ the normal velocity of $\partial\Omega(t)$. 

Finally, we demand positivity of the initial density throughout the starting domain and boundary
\begin{equation}
\rho_0 > 0 \qquad \text{in} \enskip \Omega(0) \cup \partial\Omega(0). \label{E:POSITIVEDENSITY}
\end{equation}
Our study of the Euler system with vacuum (\ref{E:VACUUM}) and a positive density (\ref{E:POSITIVEDENSITY}) was motivated by the work of Ovsyannikov \cite{ovsyannikov1956new} and Borisov-Kilin-Mamaev \cite{borisov2009hamiltonian} in which examples of global-in-time solutions were given in this context. Considering the presence of vacuum (\ref{E:VACUUM}) and our equation of state (\ref{E:EEOS}), a positive density implies formally our entropy must become unbounded and approach negative infinity towards the boundary. Unbounded entropy profiles have featured in work on black holes \cite{hsu2008black,hsu2009monsters} and an entropy diverging to negative infinity was studied in the context of dynamical systems used to model particle states \cite{evans2002comments}.

In our study of the vacuum free boundary nonisentropic Euler system with positive initial density (\ref{E:MOM})-(\ref{E:POSITIVEDENSITY}) we will make the assumption of spherical symmetry. The novel weight structure unique to this problem whereby the entropy profile will function as our changing weight leads to a fundamental loss of weight with respect to derivatives in the vorticity equation. Thus to avoid this problem, we assume radial symmetry for the Lagrangian flow map $\zeta(t,y)$
\begin{equation}
\zeta(t,y)=\chi(t,r)y, \quad r=|y|,
\end{equation}
where $\chi$ will now function as our radial flow map in Lagrangian variables. This assumption will avoid the need to have control of the problematic curl term due to loss of weight, but will introduce a coordinate singularity near the origin $r=0$ for which we employ the methodolodgy developed in the recent work \cite{guo2018continued} to handle this singularity. In Eulerian coordinates, the spherically symmetric free boundary Euler equations are
\begin{alignat}{2}
\rho(\partial_t u + u \partial_r u) + \partial_r p  & = 0 & \qquad \text{in} \enskip (0,R(t)), \label{E:MOMSS} \\    
\partial_t (r^2 \rho) + \partial_r (r^2 \rho u)  & = 0 & \quad \quad \text{in} \enskip (0,R(t)), \label{E:MSS} \\ 
\partial_t S + u \partial_r S & = 0 & \quad \quad \text{in} \enskip (0,R(t)), \label{E:SSS}
\end{alignat}
where 
\begin{equation}\label{E:EULERIANSOLNS}
\mathbf{u}(t,x)=(x/r)u(t,r), \quad \rho(t,x)=\rho(t,r), \quad S(t,x)=S(t,r), \quad r=|x|,
\end{equation}
and the moving domain $\Omega(t)=B_{R(t)}(\mathbf{0})$.

Collectively, we will study the Cauchy problem for the spherically symmetric vacuum free boundary nonisentropic Euler system with positive density. To the best of our knowledge there are no known previous global-in-time existence results for this system, with or without spherical symmetry. The main goal of this article is to construct open sets of initial data that lead to global solutions to the 
positive density nonisentropic Euler system with spherical symmetry in the presence of free vacuum boundaries.

Before we move on, we briefly discuss some known results relevant to the present article. It is well-known that the Euler equations are hyperbolic and in the whole space the existence of $C^1$ local-in-time positive density solutions follows from the theory of symmetric hyperbolic systems \cite{kato1975cauchy,majda1984compressible}. Serre \cite{Se1997} and Grassin \cite{grassin1998global} proved global existence in the whole space for a special class of initial data by perturbing solutions to the vectorial Burgers equation. Recently, Rickard \cite{rickard2020global} proved global-in-time well-posedness of the Euler equations with heat transport by the pertubation of Dyson's \cite{10.2307/24902147} isothermal affine solutions, see below for details on affine solutions. In the other direction, Sideris \cite{sideris1985} showed that singularities must form if the density is a strictly positive constant outside of a bounded set. Christodoulou-Miao~\cite{MiCr} give a thorough description of shock formation for irrotational fluids starting from smooth initial data. We refer to the works of Speck and Luk-Speck~\cite{luk2018shock,Sp} for a more general framework with respect to shock formation. Buckmaster-Shkoller-Vicol \cite{buckmaster2020shock} recently gave a constructive proof of shock formation leading to vorticity formation from an open set of initial data. We remark that these singularity and shock formation results do not apply to the physical vacuum free boundary problem.

Important examples of global-in-time solutions are given by Ovsyannikov \cite{ovsyannikov1956new}, Dyson \cite{10.2307/24902147} and Sideris \cite{SiderisSS,MR3634025}. These are the so-called affine motions which are special expanding global solutions found by a separation-of-variables ansatz for the Lagrangian flow map, see Section \ref{S:AFFINE}. Notably for our current work, the affine solutions found by Ovsyannikov \cite{ovsyannikov1956new} can be used to construct positive density solutions in the presence of vacuum boundaries, see Borisov-Kilin-Mamaev \cite{borisov2009hamiltonian}.

For the vacuum free boundary problem, local well-posedness with physical vacuum has been proven by Coutand-Shkoller \cite{Coutand2012} and Jang-Masmoudi \cite{doi:10.1002/cpa.21517}. In particular, we will adapt the method of Jang-Masmoudi \cite{doi:10.1002/cpa.21517} to obtain local well-posedness for our system. The global affine solutions of Sideris \cite{SiderisSS,MR3634025} were constructed in the free vacuum boundary setting and importantly satisfy the physical vacuum condition~\eqref{E:ENORMALDERIVATIVE}. In the isentropic (constant entropy) case,  Had\v zi\'c-Jang \cite{1610.01666} and Shkoller-Sideris~\cite{shkoller2017global} estabilished the nonlinear stability of the Sideris solutions. For the nonisentropic setting with vanishing density and bounded entropy, Rickard-Had\v zi\'c-Jang \cite{rickard2019global} proved global existence by perturbing around a rich class of nonisentropic affine motions. Finally, in the isentropic setting, Parmeshwar-Had\v zi\'c-Jang \cite{PHJ2019} do not rely on background affine solutions and obtain global existence of small density expanding solutions.

In the previous vacuum free boundary problem results, the density is vanishing at the boundary. This contrasts with the current article in which we consider a positive density. In particular, our entropy profile will instead behave like a distance function and will operate as a changing weight in our analysis.

\subsection{Spherically Symmetric Lagrangian Coordinates}\label{S:SSL}
To study the vacuum free boundary nonisentropic Euler system with positive initial density (\ref{E:MOM})-(\ref{E:POSITIVEDENSITY}) we move to Lagrangian coordinates which brings the problem onto a fixed domain. Define the flow map $\zeta$ as follows
\begin{align}\label{E:LANGRANGIANDEFN}
\partial_t \zeta (t,y) &= \mathbf{u}(t,\zeta(t,y)), \\
\zeta(0,y)&=\zeta_0(y),
\end{align} 
where $\zeta_0$ is a sufficiently smooth diffeomorphism. Introduce the notation
\begin{align}
\mathscr{A}_\zeta := [D \zeta]^{-1}, \quad &\text{(Inverse of the Jacobian matrix)}\label{E:SCRAZETA}  \\ 
\mathscr{J}_\zeta := \det[D \zeta], \quad &\text{(Jacobian determinant)} \label{E:SCRJZETA} \\
\bar{\rho}(y):=\rho_0(\zeta_0(y)) \mathscr{J}_\zeta(0,y) \quad &\text{(Lagrangian density profile)} \\
d(y):=e^{S_0(\zeta_0(y))} \quad &\text{(Langrangian entropy profile).}
\end{align}
Then it is known \cite{rickard2019global} that the nonisentropic Euler equations (\ref{E:MOM})-(\ref{E:S}) reduce to
\begin{equation}\label{E:LEULER}
\bar{\rho}  \partial_t^2 \zeta^i + (\bar{\rho}^{\gamma} d \, [\mathscr{A}_\zeta]_i^k \mathscr{J}_\zeta^{1-\gamma} )_{,k}=0.
\end{equation}
As mentioned above, we make the assumption of spherical symmetry for the remainder of this paper. Therefore we make the ansatz
\begin{equation}\label{E:SSASSUMPTION}
\zeta(t,y)=\chi(t,r)y, \quad r=|y|.
\end{equation}
Then $\mathscr{A}_\zeta$ and $\mathscr{J}_\zeta$ transform as follows \cite{jang2014nonlinear}
\begin{align}
\mathscr{A}_\zeta &=\frac{\delta_i^k}{\chi}-\frac{\chi_r y^k y^i}{\chi(\chi+\chi_r r)r}, \\
\mathscr{J}_\zeta &= \chi^2(\chi+\chi_r r). \label{E:JCHI}
\end{align}
Since we are considering the spherically symmetric case, make the assumption that $\bar{\rho}$ and $d$ are radial functions
\begin{align}
\bar{\rho}(y)&=\bar{\rho}(r), \\
d(y)&=d(r).
\end{align}
With these formulae and the fact that $\partial_k = \frac{y_k}{r} \partial_r$, by substituting (\ref{E:SSASSUMPTION}) into (\ref{E:LEULER}) we obtain
\begin{equation}\label{E:LAGRANGIANPREAFF}
\bar{\rho}  \chi_{tt} + \frac{\chi^2}{r} \partial_r \left( \bar{\rho}^{\gamma} d (\chi^2 (\chi+\chi_r r))^{-\gamma} \right)=0.
\end{equation}

\subsection{Nonisentropic Spherically Symmetric Affine Motion}\label{S:AFFINE}
By making the further ansatz $\chi(t,r)=a(t)$ for scalar $a(t)$ with $a(0)>0$, the fundamental affine ODEs found in \cite{10.2307/24902147,ovsyannikov1956new,SiderisSS} are obtained. The ansatz transforms (\ref{E:LAGRANGIANPREAFF}) into
\begin{equation}\label{E:POSTAFFANSATZ}
\bar{\rho} a_{tt} + \frac{a^{2-3\gamma}}{r} \partial_r \left(\bar{\rho}^{\gamma} d \right) = 0. 
\end{equation}
We have that $\bar{\rho}^{\gamma} d$ is independent of $t$ and hence (\ref{E:POSTAFFANSATZ}) will hold if we require the following fundamental affine ODEs to hold
\begin{align}
a_{tt}&=a^{2-3\gamma}, \label{E:AFFREQ1} \\
 \bar{\rho} r &=  - \partial_r (\bar{\rho}^{\gamma} d). \label{E:AFFREQ2}
\end{align}
Positive density motions in the presence of vacuum boundaries were obtained from the Ovsyannikov \cite{ovsyannikov1956new} affine solutions, see Borisov-Kilin-Mamaev \cite{borisov2009hamiltonian}. 
For our positive density free boundary Euler system, special affine solutions are obtained through $\bar{\rho}$ having the following properties
\begin{subequations}     
\begin{equation}\label{E:RHOZEROAFFREGULARITYA}
\bar{\rho} \in C^0[0,1] \cap C^1[0,1],
\end{equation}
\begin{equation}\label{E:RHOZEROAFFREGULARITYB}
\bar{\rho}(r) > 0 \text{ for } r \in [0,1],
\end{equation}
\begin{equation}\label{E:RHOZEROAFFREGULARITYC}
\bar{\rho}\,'(0)=0.
\end{equation} 
\end{subequations} 
Then we consider the following form for $d$ which solves (\ref{E:AFFREQ2})
\begin{equation}\label{E:EXPLICITFORMFORD}
d(r)=\frac{  \int_{r}^1 \ell \bar{\rho} (\ell) \, d \ell }{(\bar{\rho}(r))^\gamma}.
\end{equation}
With (\ref{E:EXPLICITFORMFORD}), we have the fundamental property $d(1)=0$ which is crucial to permit a solution to our positive density Euler system. Our next calculation shows why the vacuum boundary condition (\ref{E:ENORMALDERIVATIVE}) is satisfied by our affine solution
\begin{equation}
(\bar{\rho}^{\gamma-1} d)'(1)=\lim_{r \rightarrow 1^-} \frac{ \int_r^1 \ell \bar{\rho}(\ell) \, d \ell }{\bar{\rho}(r)(r-1)}=\lim_{r \rightarrow 1^-} \frac{ - r \bar{\rho}(r)  }{\bar{\rho}'(r)(r-1)+\bar{\rho}(r)} = -1,
\end{equation}
where we have used L'Hospital's Rule.
Then with $a \in C(\mathbb{R},\mathbb{R}^+) \cap C^\infty(\mathbb{R},\mathbb{R})$ solving the ODE (\ref{E:AFFREQ1}), $\bar{\rho}$ satisfying (\ref{E:RHOZEROAFFREGULARITYA})-(\ref{E:RHOZEROAFFREGULARITYC}) and finally $d$ being given by (\ref{E:EXPLICITFORMFORD}), the associated solution of the positive density Euler system is
\begin{align}
\mathbf{u}_a(t,x)&=\frac{a'(t)x}{a(t)}, \\
\rho_a(t,x)&=\frac{\bar{\rho}(r)}{a(t)}, \label{E:AFFINEDENSITY} \\
S_a(t,x) & = \ln d(r). \label{E:AFFENTROPY}
\end{align}
noting that $r=|\tfrac{x}{a(t)}|$ here. 
\begin{remark}[Unbounded Entropy]
From (\ref{E:AFFENTROPY}), the property $\lim_{r \rightarrow 1^-} d(r) =d(1)= 0$ corresponds to our affine entropy approaching negative infinity as it approaches the boundary.
\end{remark}

Now at this stage, we consider the profiles $\bar{\rho}$ of the form
\begin{align}
\bar{\rho}(r)&=\phi(r), \label{E:BARRHODEMAND}
\end{align}
where    
$\phi \in C^{k}[0,1], \ \phi>0$ satisfying $\phi'(0)=0$ with $k\in \mathbb N$ to be specified. Here we demand the condition  $\phi'(0)=0$ to ensure the regularity of $\accentset{\circ}{\rho}$ at the center as in \eqref{E:RHOZEROAFFREGULARITYC}.
Then from (\ref{E:EXPLICITFORMFORD}), we immediately have $d \in C^{k}[0,1]$ also.

Finally, we note for all $\gamma > 1$ we have the asymptotics
\begin{equation}
a(t) \sim 1+t, \, t \geq 0.
\end{equation}
This follows from Theorem 3 \cite{MR3634025} if we consider the special case of the full 3D problem where $A(t)=\text{diag}(a(t),a(t),a(t))$.

We denote the set of affine motions under consideration by $\mathscr{S}$. To recap, the set $\mathscr{S}$ is parametrized by the quadruple
\begin{equation}    
(a(0),a'(0),\phi) \in \mathbb{R}_+ \times \mathbb{R}  \times \mathcal Z_k,
\end{equation}
where
\begin{equation}\label{E:PHIDEMAND}
\mathcal Z_k := \left\{ \phi \in C^k[0,1]: \phi>0, \  \phi'(0)= 0 \right \}
\end{equation}
and we take $k \in \mathbb{N}$ sufficiently large (to be specified later in Theorems \ref{T:LWP} and \ref{T:MAINTHEOREMGAMMALEQ5OVER3}).

\begin{remark}[Eulerian description of spherically symmetric solutions]
At this stage it is worth giving the connections between the Eulerian description of spherically symmetric solutions (\ref{E:EULERIANSOLNS}), spherically symmetric Lagrangian coordinates and background affine motion. For the velocity, using (\ref{E:LANGRANGIANDEFN}) and (\ref{E:SSASSUMPTION}),
\begin{equation}
\mathbf{u}(t,\zeta(t,y))=\partial_t \zeta(t,y) = \chi_t(t,r) y = (y/r) (r \chi_t (r,t)) \; \; \text{so that} \; \;
u(t,|\zeta(t,y)|)=r \chi_t (t,r).
\end{equation}
For the density
\begin{equation}
\rho(t,|\zeta(t,y)|)= \frac{\bar{\rho}(r)}{\chi^2(\chi+\chi_r r)}.
\end{equation}
For the entropy
\begin{equation}
S(t,|\zeta(t,y)|)= \ln d(r).
\end{equation}
\end{remark}
\begin{remark}[Eulerian initial density $\rho_0$ and entropy $S_0$]
The spherically symmetric Eulerian initial density $\rho_0$ and entropy $S_0$ are connected to the background affine motion via
\begin{align}
\rho_0(|x|)&=\frac{\bar{\rho}(|\zeta_0^{-1} (x)|)}{[\chi_0(|\zeta_0^{-1} (x)|)]^2 (\chi_0(|\zeta_0^{-1} (x)|) +\chi_0'(|\zeta_0^{-1} (x)|) |\zeta_0^{-1} (x)|)} \notag \\
S_0(|x|)&=\ln d( | \zeta_0^{-1} (y) | ),
\end{align}
where $\chi_0'(r)=\partial_r \chi_0 (r)$.
\end{remark}

\section{Formulation and Main Global Existence Result}
\subsection{Perturbation of Affine Motion}
We derive the equation for the perturbation of our affine motion. With an affine motion $a$ fixed, define the modified flow map $\xi=\frac{\chi}{a}$. We note considering the Lagrangian motion given by $\xi(t,r)y$ then the $\mathscr{J}_\xi$ equivalent of $\mathscr{J}_\zeta$ is given by
\begin{equation}
\mathscr{J}_\xi = \xi^2(\xi+\xi_r r):=\mathscr{J}
\end{equation}
Thus $\mathscr{J}_\xi = a^3 \mathscr{J}_\zeta$ from (\ref{E:JCHI}). Now from (\ref{E:LAGRANGIANPREAFF}) we have
\begin{equation}
\bar{\rho}\left(a \xi_{tt} + 2 a_t \xi_t + a_{tt} \xi\right) +   a^{2-3\gamma}\frac{\xi^2}{r} \partial_r\left( \bar{\rho}^{\gamma} d \J^{-\gamma}) \right)=0. 
\end{equation}
Apply (\ref{E:AFFREQ1}) and multiply by $a^{3\gamma-2}$
\begin{equation}\label{E:ETAPRETAU}
\bar{\rho}\left(a^{3\gamma-1} \xi_{tt} +  a^{3\gamma-2} a_t \xi_t + \xi\right) + \frac{\xi^2}{r} \partial_r\left( \bar{\rho}^{\gamma} d \J^{-\gamma}) \right)=0. 
\end{equation}
Next make a change of time variable by setting 
$$\frac{d \tau}{dt}=\frac{1}{a}.$$
Then we can formulate (\ref{E:ETAPRETAU}) as  
\begin{equation}\label{E:ETAPOSTTAU}
\bar{\rho}\left(a^{3\gamma-3} \xi_{\tau \tau} +  a^{3\gamma-4} a_\tau \xi_\tau + \xi \right) + \frac{\xi^2}{r} \partial_r\left( \bar{\rho}^{\gamma} d \J^{-\gamma}) \right)=0. 
\end{equation}
Note $\xi(r) \equiv 1$ corresponds to affine motion. Introducing the perturbation
\begin{equation}
\uptheta(\tau,r):=\xi(\tau,r)-1,
\end{equation}
equation (\ref{E:ETAPOSTTAU}) can be written in terms of $\uptheta$ 
\begin{equation}\label{E:THETA1}
\bar{\rho}\left(a^{3\gamma-3} \uptheta_{\tau \tau} +  a^{3\gamma-4} a_\tau \uptheta_\tau \right)
+ \bar{\rho} (1+\uptheta) +  \frac{(1+\uptheta)^2}{r} \partial_r \left( \bar{\rho}^{\gamma} d \J^{-\gamma} \right)=0.
\end{equation}
Finally apply (\ref{E:AFFREQ2}) and note $(1+\uptheta)-(1+\uptheta)^2=-(1+\uptheta)\uptheta$ to obtain
\begin{equation}\label{E:THETAGAMMALEQ3}
\bar{\rho} \left(a^{3\gamma-3} \uptheta_{\tau \tau} +  a^{3\gamma-4} a_\tau \uptheta_\tau\right) -   \bar{\rho} \uptheta (1+\uptheta) +  \frac{(1+\uptheta)^2}{r} \partial_r \left( \bar{\rho}^{\gamma} d \,(\mathscr{J}^{-\gamma}-1 )\right)=0.
\end{equation}

\subsection{The $H$ Equation}
We derive the equation to be used in our estimates that will help us overcome the coordinate singularity near the origin $r=0$. Divide (\ref{E:THETAGAMMALEQ3}) by $\bar{\rho}$
\begin{equation}\label{E:DIVIDETHETAEQNBARRHO}
a^{3 \gamma-3} \uptheta_{\tau \tau} + a^{3 \gamma-4} a_{\tau} \uptheta_\tau - \uptheta(1+\uptheta) + \frac{(1+\uptheta)^2}{\bar{\rho} r} \partial_r ( \bar{\rho}^\gamma d (\mathscr{J}^{-\gamma} - 1)) = 0.
\end{equation}
Note that using the affine ODE $\bar{\rho} r =  - \partial_r (\bar{\rho}^{\gamma} d)$ we can write
\begin{equation}\label{E:EXPANDTHETAGAMMALEQ3LASTTERM}
\frac{1}{\bar{\rho}} r \partial_r [ (\bar{\rho}^\gamma d) (\mathscr{J}^{-\gamma}-1)] = \bar{\rho}^{\gamma-1} d\, r \partial_r (\mathscr{J}^{-\gamma}) -r^2 (\mathscr{J}^{-\gamma}-1).
\end{equation}
Now
\begin{equation}
\partial_r (\mathscr{J}^{-\gamma})=-\gamma \mathscr{J}^{-\gamma-1} \partial_r \J,
\end{equation}
and 
\begin{align}
\partial_r \mathscr{J} & = \partial_r (\xi^2(\xi+\xi_r r)) \notag \\
&= \partial_r ( \xi^2 (1+\uptheta+\uptheta_{r} r )) \notag \\
&= \xi^2 r \uptheta_{r r} + (2 \xi^2 + 2 \xi \xi_r r)\uptheta_r + (2 \xi \xi_r)\uptheta + 2 \xi \xi_r \notag \\
&=\xi^2 r \uptheta_{r r} + (2 \xi^2 + 2 \xi \xi_r r + 2 \xi) \uptheta_r  + (2 \xi \xi_r) \uptheta
\end{align}
where we recall
\begin{equation}
\uptheta=\xi-1
\end{equation}
and hence $\uptheta_r = \xi_r$. Thus
\begin{align}
&\bar{\rho}^{\gamma-1} d \, r \partial_r (\mathscr{J}^{-\gamma} )  \notag \\
&= -\gamma \bar{\rho}^{\gamma-1} d \xi^2 r^2 \uptheta_{r r} \mathscr{J}^{-\gamma-1} -\gamma \bar{\rho}^{\gamma-1} d \, r (2 \xi^2 + 2 \xi \xi_r r + 2 \xi ) \uptheta_r \mathscr{J}^{-\gamma-1} - \gamma \bar{\rho}^{\gamma-1} d \, r (2 \xi \xi_r ) \uptheta \mathscr{J}^{-\gamma-1}.
\end{align}
Then from (\ref{E:EXPANDTHETAGAMMALEQ3LASTTERM})
\begin{align}
& \frac{1}{\bar{\rho}} ( r \partial_r ) ( ( \bar{\rho}^\gamma d ) (\mathscr{J}^{-\gamma} -1 ) ) \notag \\
&= -\gamma \bar{\rho}^{\gamma-1} d \xi^2 r^2 \uptheta_{r r} \mathscr{J}^{-\gamma-1} - \gamma \bar{\rho}^{\gamma-1} r \, d ( 2 \xi^2 + 2 \xi \xi_r r + 2 \xi  ) \mathscr{J}^{-\gamma-1} \uptheta_r \notag \\
&- \gamma \bar{\rho}^{\gamma-1} d \, r (2 \xi \xi_r ) \uptheta \mathscr{J}^{-\gamma-1} - r^2 (\mathscr{J}^{-\gamma} - 1) \notag \\
&=-\gamma r \mathscr{J}^{-\gamma-1} \xi^2 \frac{1}{\bar{\rho}} \partial_r (\bar{\rho}^{\gamma} d \frac{1}{r^2} \partial_r (r^3 \uptheta)) - r^2 ( (\mathscr{J}^{-\gamma} - 1)+\gamma \mathscr{J}^{-\gamma-1} \xi^2 [ r \uptheta_r + 3 \uptheta] ) \notag \\
&- \gamma \bar{\rho}^{\gamma-1} r \, d \mathscr{J}^{-\gamma-1} ( 2 \xi \xi_r r + 2 \xi - 2 \xi^2 ) \uptheta_r - \gamma \bar{\rho}^{\gamma-1} d \, r 2 \xi \xi_r \mathscr{J}^{-\gamma-1} \uptheta \label{E:APPLYEXPANDTHETAGAMMALEQ3LASTTERM}
\end{align}
We notice that
\begin{equation}
2 \xi - 2 \xi^2 \uptheta_r =  2 \xi ( 1- \xi ) \xi_r =  - 2 \xi \xi_r \uptheta,
\end{equation}
which leads to a cancellation on the last line of (\ref{E:APPLYEXPANDTHETAGAMMALEQ3LASTTERM}). Also
\begin{align}
\J^{-\gamma} - 1 &= -\gamma (\J - 1) + ( (\J^{-\gamma} - 1) + \gamma ( \J - 1 )) \\
\gamma \mathscr{J}^{-\gamma-1} \xi^2 ( r \uptheta_r + 3 \uptheta)  & = \gamma \xi^2 (r \uptheta_r + 3 \uptheta ) + \gamma ( \J^{-\gamma-1} -1 ) \xi^2 ( r \partial_r \uptheta + 3 \uptheta). 
\end{align}
Then
\begin{align}
&(\mathscr{J}^{-\gamma} - 1)+\gamma \mathscr{J}^{-\gamma-1} \xi^2 ( r \uptheta_r + 3 \uptheta ) \notag \\
&= -\gamma ( \J - 1) + \gamma \xi^2 (r \partial_r \uptheta + 3 \uptheta )  +  ( (\J^{-\gamma} - 1) + \gamma ( \J - 1 )) + \gamma ( \J^{-\gamma-1} - 1) \xi^2 ( r \partial_r \uptheta + 3 \uptheta).
\end{align}
Now using 
\begin{equation}
\mathscr{J}=\xi^2(\xi+\xi_r r)=(1+\uptheta)^2(1+\uptheta+\uptheta_r r)=1+\frac{1}{r^2} \left( r^3 \left(\uptheta + \uptheta^2 + \frac{\uptheta^3}{3} \right) \right)_r, \label{E:JEXPAND}
\end{equation}
we have
\begin{align}
\mathscr{J}-1&=\frac{1}{r^2} \left( r^3 \left(\uptheta + \uptheta^2 + \frac{\uptheta^3}{3} \right) \right)_r \notag \\
&=\frac{1}{r^2}(3 r^2 (\uptheta+\uptheta^2 + \frac{\uptheta^3}{3}) + r^3 (\uptheta_r + 2 \uptheta \uptheta_r + 3 \uptheta^2 \uptheta_r) \notag \\
&=3 \uptheta + 3 \uptheta^2+\uptheta^3 +r \uptheta_r + 2 \uptheta \uptheta_r + r \uptheta^2 \uptheta_r.
\end{align}
On the other hand, notice that
\begin{align}
\xi^2 ( r \partial_r \uptheta + 3 \uptheta) &= (1+\uptheta)^2 (r \partial_r \uptheta + 3 \uptheta) \notag \\
&=r \partial_r \uptheta + 3 \uptheta + 2 r \uptheta \partial_r \uptheta + \uptheta^2 r \partial_r \uptheta + 3 \uptheta^3 + 6 \uptheta^2.
\end{align}
Hence canceling terms 
\begin{equation}
-\gamma ( \J - 1) + \gamma \xi^2 (r \partial_r \uptheta + 3 \uptheta ) = -\gamma(-3 \uptheta^2-2\uptheta^3).
\end{equation}
Therefore we now have
\begin{align}
&\frac{\xi^2}{\bar{\rho} r^2} (r \partial_r) ( \bar{\rho}^\gamma d (\mathscr{J}^{-\gamma}-1)) = - \gamma \frac{ \xi^4 }{\J^{\gamma+1} r  \bar{\rho} } \partial_r  ( \bar{\rho}^\gamma d \tfrac{1}{r^2} \partial_r (r^3 \uptheta)) \notag \\
& - \gamma \bar{\rho}^{\gamma-1} d \J^{-\gamma-1} 2 \xi^3 \xi_r \uptheta_r \notag \\
& - \xi^2 ( (\J^{-\gamma}-1)+\gamma(\J-1) - \gamma(-3 \uptheta^2 - 2 \uptheta^3)  + \gamma (\J^{-\gamma-1} -1 ) \xi^2 (r \partial_r \uptheta + 3 \uptheta)). \label{E:PREREMAINDERFORM}
\end{align}
Rewrite the second line of (\ref{E:PREREMAINDERFORM}) as follows
\begin{equation}
- \gamma \bar{\rho}^{\gamma-1} d \J^{-\gamma-1} 2 \xi^3 \xi_r \uptheta_r = - 2 \gamma \bar{\rho}^{\gamma-1} d \frac{\xi^3}{\J^{\gamma+1} r^2} (r \partial_r \uptheta)^2
\end{equation}
Then let
\begin{align}
\mathcal{R}_1[\uptheta]&:=- 2 \gamma \bar{\rho}^{\gamma-1} d \frac{\xi^3}{\J^{\gamma+1} r^2} (r \partial_r \uptheta)^2 \\
\mathcal{R}_2[\uptheta]&:=- \xi^2 ( (\J^{-\gamma}-1)+\gamma(\J-1) - \gamma(-3 \uptheta^2 - 2 \uptheta^3)  + \gamma (\J^{-\gamma-1} -1 ) \xi^2 (r \partial_r \uptheta + 3 \uptheta)).
\end{align}
Therefore (\ref{E:DIVIDETHETAEQNBARRHO}) can be written as
\begin{equation}\label{E:THETAPREH}
a^{3 \gamma-3} \uptheta_{\tau \tau} + a^{3 \gamma-4} a_{\tau} \uptheta_\tau - \gamma \frac{ \xi^4 }{\J^{\gamma+1} r  \bar{\rho} } \partial_r  ( \bar{\rho}^\gamma d \tfrac{1}{r^2} \partial_r (r^3 \uptheta)) -  \uptheta(1+\uptheta) + \mathcal{R}_1[\uptheta] + \mathcal{R}_2[\uptheta]  = 0.
\end{equation}
Next let $H:=r \uptheta$. So $\uptheta=\frac{H}{r}$. Then $H$ solves
\begin{align}\label{E:HEQN}
&a^{3 \gamma-3} H_{\tau \tau} + a^{3 \gamma-4} a_{\tau} H_\tau  - \gamma \frac{ \xi^4 }{\J^{\gamma+1}  \bar{\rho} } \partial_r  ( \bar{\rho}^\gamma d \tfrac{1}{r^2} \partial_r (r^2 H)) -  H\left(1+\frac{H}{r}\right) \notag \\
& \qquad + r \mathcal{R}_1[\tfrac{H}{r}] + r \mathcal{R}_2[\tfrac{H}{r}]  = 0,
\end{align}
with the initial conditions 
\begin{equation}\label{E:THETAICGAMMALEQ5OVER3}
H(0,y)=H_0(r), \quad H_\tau(0,r)=\partial_\tau H_0(r), \quad (r \in [0,1]).
\end{equation}

\subsection{Notation}\label{S:NOTATION}
First introduce the radial equivalent of the three-dimensional divergence operator
\begin{equation}
D_r:=\frac{1}{r^2} \partial_r (r^2 \cdot)
\end{equation}
To avoid singularities at $r=0$ when applying high-order derivatives, define
\begin{equation}
\mathcal{D}_j := \begin{cases}
(\partial_r D_r)^{\tfrac{j}{2}} & \text{ if } \  j \text { is even} \\
D_r (\partial_r D_r)^{\tfrac{j-1}{2}} & \text{ if } \  j \text{ is odd}
\end{cases}
\end{equation}
and set $\mathcal{D}_0 = 1$. Also define
\begin{equation}
\bar{\mathcal{D}}_i :=\begin{cases}
\mathcal{D}_0 & \text{ for } \  i =0 \\
\mathcal{D}_{i-1} \partial_r & \text{ for } \ i \geq 1
\end{cases}.
\end{equation}
We will use the following elliptic operators to derive high-order equations
\begin{align}
L_k f &:=\frac{1}{\bar{\rho}^{1+k(\gamma-1)} d^k} \partial_r \left( \bar{\rho}^{\gamma+k(\gamma-1)} d^{1+k} D_r f \right) \label{E:LK} \\
L_k^* h &:= \frac{1}{\bar{\rho}^{1+k(\gamma-1)} d^k} D_r \left( \bar{\rho}^{\gamma+k(\gamma-1)} d^{1+k} \partial_r h \right). \label{E:LKSTAR}
\end{align}
Then define
\begin{equation}
 \mathcal L_{j} \mathcal D_j : = 
\begin{cases}
L_{j} \mathcal D_j & \text{ if $j$ is even}\\
L^\ast_{j}\mathcal D_j & \text{ if $j$ is odd}
\end{cases}.
\end{equation}
Now for any $k \in \mathbb{Z}_{\geq 0}$, we consider the weighted $L^2$ norm
\begin{equation}
\| f \|^2_{k} := \int_0^1 d^k (f)^2 r^2 \, dr
\end{equation}
The smooth cut-off function $\psi \geq 0$ such that
\begin{equation}
\psi=1  \text{ on } [0,\tfrac{1}{2}]\text{, } \psi = 0 \text{ on } [\tfrac{3}{4},1] \text{ and } \psi' \leq 0,
\end{equation}
will be useful. The following vector fields will be important in obtaining high-order estimates successfully taking into account the coordinate singularity near the origin
\begin{align}
\mathcal{P}_{2 j +2}&:=\left\lbrace  \prod_{k=1}^{j+1} \partial_r V_k : V_k \in \{ D_r, \tfrac{1}{r} \} \right\rbrace , \label{E:PEVEN} \\
\mathcal{P}_{2 j +1}&:= \left\lbrace  V_{j+1} \prod_{k=1}^{j} \partial_r V_k : V_k \in \{ D_r,\tfrac{1}{r} \} \right\rbrace , \label{E:PODD}
\end{align}
for $j \geq 0$, and set $\mathcal{P}_0 = \{ 1 \}$. Also define 
\begin{align}
\overline{\mathcal{P}}_{2j+2}&:=\{ W \partial_r : W \in \mathcal{P}_{2 j+1} \} \notag \\
\overline{\mathcal{P}}_{2j+1}&:=\{ W \partial_r : W \in \mathcal{P}_{2 j} \},
\end{align}
for $j \geq 0$, and set $\overline{\mathcal{P}}_0=\{ 1 \}$.

\subsection{High-order Norm}
Our time weights will differ depending on whether $\gamma \in (1,\frac53]$ or $\gamma > \frac53$. This is because we take a slightly different approach for $\gamma > \frac53$ by an adaptation of \cite{shkoller2017global}, applied to our spherically symmetric nonisentropic setting.

On this note, introduce the following $\gamma$ dependent exponents
\begin{equation}
d(\gamma):=\begin{cases}
3\gamma - 3 & \text{ if } \  1<\gamma\leq \frac53 \\
2 & \text{ if } \  \gamma>\frac53
\end{cases}; \quad b(\gamma):=d(\gamma)+3-3\gamma=\begin{cases}
0 & \text{ if } \  1<\gamma\leq \frac53 \\
5-3 \gamma & \text{ if } \  \gamma>\frac53
\end{cases}.
\end{equation} 

Let $N \in \mathbb{N}$. To measure the size $H$, we define the high-order weighted Sobolev norm as follows
\begin{equation}\label{E:SNNORM}
\mathcal{S}^N(H,H_\tau)=\mathcal{S}^N(\tau):=\sup_{0 \leq \tau' \leq \tau} \left\lbrace \sum_{i=0}^N  a^{d(\gamma)} \| \mathcal{D}_i H_\tau \|_i^2 + \sum_{i=0}^{N-1} \| \mathcal{D}_{i+1} H \|_{i+1}^2 + \sum_{i=N} a^{b(\gamma)} \| \mathcal{D}_{i+1} H \|_{i+1}^2 \right\rbrace.
\end{equation}

\subsection{Main Theorem}
Before giving our main theorem, first define the important $a(\tau)$ related quantities
\begin{equation}\label{E:MU1MU0DEFNGAMMALEQ5OVER3}
a_1:=\lim_{\tau \rightarrow \infty} \frac{a_\tau(\tau)}{a(\tau)}, \quad a_0:=\frac{d(\gamma)}{2}a_1.
\end{equation}

\textbf{Local Well-Posedness.} Next, we give the local well-posedness of our system.
\begin{theorem}\label{T:LWP}
Suppose $\gamma > 1$. Fix $N \geq 8$. Let $k \geq N$ in (\ref{E:PHIDEMAND}). Then there are $\epsilon_0>0$, $\lambda_0 > 0$ and $T>0$ such that for every $\epsilon \in (0,\epsilon_0]$, $\lambda \in (0,\lambda_0]$ and pair of initial data for (\ref{E:HEQN}) $(H_0,\partial_\tau H_0)$ satisfying $\mathcal{S}^N(H_0,\partial_\tau H_0) \leq \epsilon$ and $\| H(0) \|^2_0 \leq \lambda$, there exists a unique solution $(H(\tau),H_\tau(\tau)):[0,1] \rightarrow \mathbb{R} \times \mathbb{R}$ to (\ref{E:HEQN})-(\ref{E:THETAICGAMMALEQ5OVER3}) for all $\tau \in [0,T]$. The solution has the property $\mathcal{S}^N(H,H_\tau) \lesssim \epsilon$ for each $\tau \in [0,T]$. Furthermore, the map $[0,T]\ni\tau\mapsto\mathcal{S}^N(\tau)\in\mathbb R_+$ is continuous. 
\end{theorem}
We give the sketch of the proof of Theorem \ref{T:LWP} in Appendix \ref{A:LWP}. \\

\textbf{A priori assumptions.} Finally before our main theorem, make the following a priori assumptions on our local solutions from Theorem \ref{T:LWP}
\begin{align}
\mathcal{S}^N(\tau) &< \frac{1}{3}, \label{E:SNAPRIORI} \\
|\J - 1 | &< \frac{1}{3}, \label{E:JAPRIORI} \\
| \partial_r (\uptheta) | = | \partial_r \left( \frac{H}{r} \right) | &< \frac{1}{3}, \label{E:PARTIALRTHETAAPRIORI}  \\
| \partial_r^2 (\uptheta) | = | \partial_r^2 \left( \frac{H}{r} \right) | &< \frac{1}{3}. \label{E:DOUBLEPARTIALRTHETAAPRIORI}
\end{align}
\begin{remark} 
Using $\J=\xi^2( \xi + \xi_r r)$ and $\uptheta_r = \xi_r$, (\ref{E:JAPRIORI})-(\ref{E:PARTIALRTHETAAPRIORI}) imply
\begin{equation}\label{E:XIBOUNDEDABOVEANDBELOW}
1 \lesssim | \xi | \lesssim 1.
\end{equation}
Furthermore, using $\J_r = ( \xi^3 + \xi^2 \xi_r r)_r $ and $\uptheta_r = \xi_r$, (\ref{E:PARTIALRTHETAAPRIORI})-(\ref{E:XIBOUNDEDABOVEANDBELOW}) imply
\begin{equation}
1 \lesssim | \J_r | \lesssim 1.
\end{equation}
\end{remark}
We are now ready to give our main theorem.

\begin{theorem}\label{T:MAINTHEOREMGAMMALEQ5OVER3}\label{T:MAINTHEOREMGAMMALEQ5OVER3}
Suppose $\gamma >1$. Fix $N \geq 8$. Let $k \geq N$. Consider a fixed quadruple
\begin{equation}
(a(0),a'(0),\phi)\in \mathbb{R}_+ \times \mathbb{R} \times \mathcal Z_k, 
\end{equation} 
parametrizing a nonisentropic affine motion from the set $\mathscr{S}$. 
Then there are $\epsilon_0>0$ and $\lambda_0 > 0$ such that for every $\epsilon \in (0,\epsilon_0]$, $\lambda \in (0,\lambda_0]$ and pair of initial data for (\ref{E:HEQN}) $(H_0,\partial_\tau H_0)$ satisfying $\mathcal{S}^N(H_0,\partial_\tau H_0) \leq \epsilon$ and $\| H(0) \|^2_0 \leq \lambda$, there exists a global-in-time solution, $(H,H_\tau)$, to the initial value problem (\ref{E:HEQN})-(\ref{E:THETAICGAMMALEQ5OVER3}) and a constant $C>0$ such that
\begin{equation}
\mathcal S^N(H,H_\tau)(\tau)  \le C (\varepsilon+\lambda), \ \ 0\le\tau<\infty. \label{E:GLOBALBOUND} 
\end{equation} 
\end{theorem}
We believe Theorem \ref{T:MAINTHEOREMGAMMALEQ5OVER3} is the first global existence result for the positive density vacuum boundary Euler system.

Henceforth we assume we are working with a unique local solution $(H,H_\tau):[0,1] \rightarrow \mathbb{R} \times \mathbb{R}$ to (\ref{E:HEQN})-(\ref{E:THETAICGAMMALEQ5OVER3}) such that $\mathcal{S}^N(H) < \infty$ on $[0,T]$ with $T > 0$ fixed: Theorem \ref{T:LWP} ensures the existence of such a solution, and furthermore we assume this local solution satisfies the a priori assumptions (\ref{E:SNAPRIORI})-(\ref{E:DOUBLEPARTIALRTHETAAPRIORI}).

To prove our main result, we apply weighted energy estimates. A similar methodology to \cite{guo2018continued} in conjunction with a weighted energy estimate method similar to that used in \cite{rickard2019global} allows us to simultaneously handle the coordinate singularity near the origin, the exponentially growing-in-time coefficients and the vacuum boundary. Firstly the particular choice of derivative operators in combination with the introduction of special vector field classes, see Section \ref{S:NOTATION}, allows us to circumvent the coordinate singularity at $r=0$. Secondly the exponentially growing time weights take advantage of the stabilizing effect of the expanding background affine motion. Finally the increase in spatial weight $d$ in accordance with an increase in derivatives will be essential in avoiding potentially dangerous negative powers of $d$ near the boundary.

It is worth noting that the number of derivatives required to close estimates and prove the main theorem here does not depend on $\gamma$. This is because the weight structure involving $d$ does not depend on $\gamma$. This contrasts to previous free boundary works using weighted estimates with a vanishing density, see \cite{1610.01666,rickard2019global,shkoller2017global} for example.

Coercivity estimates are employed to account for the fact that our equation structure does not include a zeroth order contribution of $H$ which is seen in our definition of $\mathcal{S}^N$. Furthermore, we also use coercivity estimates to obtain results for all $\gamma>1$  because of the time weight manipulation necessary for $\gamma > \frac{5}{3}$, see Section \ref{S:TWM}.

Immediately below in Section \ref{S:ENERGYESTIMATES} we prove our high order energy estimates and in Section \ref{S:PROOFOFMAINTHEOREM} we prove our Main Theorem \ref{T:MAINTHEOREMGAMMALEQ5OVER3} by means of a continuity argument.

\section{Energy Estimates}\label{S:ENERGYESTIMATES}
\subsection{Differentiated Equation}
Let $1 \leq i \leq n$. Apply $\mathcal{D}_i$ to our equation for $H$ (\ref{E:HEQN}) 
\begin{align}
&a^{3 \gamma-3} \mathcal{D}_i H_{\tau \tau} + a^{3 \gamma-4} a_{\tau} \mathcal{D}_i H_\tau - \mathcal{D}_i H -  \mathcal{D}_i \left[\frac{H^2}{r}\right] - \gamma \mathcal{D}_i \left[ \frac{ \xi^4 }{\J^{\gamma+1}  \bar{\rho} } \partial_r  ( \bar{\rho}^\gamma d \tfrac{1}{r^2} \partial_r (r^2 H)) \right] \notag \\
&+ \mathcal{D}_i \left[ r \mathcal{R}_1[\tfrac{H}{r}] + r \mathcal{R}_2[\tfrac{H}{r}]\right]   = 0. \label{E:HEQNAPPLYDI}
\end{align}
Using the product rule for $\mathcal{D}_i$ Lemma \ref{L:PRODUCTRULEDI}, we compute
\begin{align}
& \mathcal{D}_i \left[ \frac{ \xi^4 }{\J^{\gamma+1}  \bar{\rho} } \partial_r  ( \bar{\rho}^\gamma d \tfrac{1}{r^2} \partial_r (r^2 H)) \right] = \mathcal{D}_i \left[ \frac{1}{\bar{\rho}} \partial_r  ( \bar{\rho}^\gamma d \tfrac{1}{r^2} \partial_r (r^2 H)) \right] \frac{\xi^4}{\J^{\gamma+1}} \notag \\
&+\bar{\mathcal{D}}_{i-1} \left[ \partial_r \left( \frac{\xi^4}{\J^{\gamma+1}} \right) \frac{1}{\bar{\rho}} \partial_r ( \bar{\rho}^\gamma d \tfrac{1}{r^2} \partial_r ( r^2 H ) ) \right] \notag \\
&+ \bar{\mathcal{D}}_{i-1} \left[ \frac{\xi^4}{\J^{\gamma+1}} D_r \left( \frac{1}{\bar{\rho}} \partial_r ( \bar{\rho}^\gamma d \tfrac{1}{r^2} \partial_r ( r^2 H ) ) \right) \right] -\frac{\xi^4}{\J^{\gamma+1}} \left( \bar{\mathcal{D}}_{i-1} D_r \left( \frac{1}{\bar{\rho}} \partial_r ( \bar{\rho}^\gamma d \tfrac{1}{r^2} \partial_r ( r^2 H ) ) \right) \right]. \label{E:APPLYPRODUCTRULE}
\end{align}
The last line of (\ref{E:APPLYPRODUCTRULE}) can be written using the commutator as follows
\begin{equation}
[\bar{\mathcal{D}}_{i-1},\frac{\xi^4}{\J^{\gamma+1}}] D_r \left( \frac{1}{\bar{\rho}} \partial_r ( \bar{\rho}^\gamma d \tfrac{1}{r^2} \partial_r ( r^2 H ) ) \right).
\end{equation}
Next we compute $\mathcal{D}_i \left[ \frac{1}{\bar{\rho}} \partial_r  ( \bar{\rho}^\gamma d \tfrac{1}{r^2} \partial_r (r^2 H)) \right]$. To this end, introduce
\begin{align}
L_k f &:=\frac{1}{\bar{\rho}^{1+k(\gamma-1)} d^k} \partial_r \left( \bar{\rho}^{\gamma+k(\gamma-1)} d^{1+k} D_r f \right) \\
L_k^* h &:= \frac{1}{\bar{\rho}^{1+k(\gamma-1)} d^k} D_r \left( \bar{\rho}^{\gamma+k(\gamma-1)} d^{1+k} \partial_r h \right)  \\
 \mathcal L_{j} \mathcal D_j &: = 
\begin{cases}
L_{j} \mathcal D_j & \text{ if $j$ is even}\\
L^\ast_{j}\mathcal D_j & \text{ if $j$ is odd}
\end{cases}. 
\end{align}
Recalling the notation $L_k$ (\ref{E:LK}) and $L_k^*$ (\ref{E:LKSTAR}), first
\begin{align}
& D_r L_k f = L_{1+k}^* D_r f + [ (\gamma+k(\gamma-1)) D_r (\bar{\rho}^{\gamma-2} \bar{\rho}_r d) + (1+k) D_r (d _r \bar{\rho}^{\gamma-1}) ] D_r f \notag \\
&= L_{1+k}^* D_r f + Q_{+} D_r f
\end{align}
where we define
\begin{align}
Q_{+}&:=\tfrac1r ( (2 ( \gamma + k (\gamma-1))\bar{\rho}^{\gamma-2} \bar{\rho}_{r} d + (2 (1+k) \bar{\rho}^{\gamma-1}) d_r) + \notag \\
&+((\gamma+k(\gamma-1))((\gamma-2)\bar{\rho}^{\gamma-3} \bar{\rho}_r^2 + \bar{\rho}^{\gamma-2} \bar{\rho}_{r r})d +(2+2k) (\gamma-1) \bar{\rho}^{\gamma-2}\bar{\rho}_r d_r + (1+k) \bar{\rho}^{\gamma-1} d_{r r}.
\end{align}
Also
\begin{align}
& \partial_r L_k^* h = L_{1+k} \partial_r h + \left[ \partial_r \left( \frac{\partial_r (\bar{\rho}^{\gamma+k(\gamma-1)} d^{1+k})}{\bar{\rho}^{1+k(\gamma-1)} d^k} \right) - \frac{2}{r} \frac{\partial_r (\bar{\rho}^{\gamma+k(\gamma-1)} d^{1+k})}{\bar{\rho}^{1+k(\gamma-1)} d^k} \right] \partial_r h \notag \\
&= L_{1+k} \partial_r h + Q_{-} \partial_r h.
\end{align}
where we define
\begin{align}
Q_{-}&:=(\gamma+k(\gamma-1))((\gamma-2) \bar{\rho}^{\gamma-3} \bar{\rho}_r^2 d + \bar{\rho}^{\gamma-2}\bar{\rho}_{rr} d + \bar{\rho}^{\gamma-2} \bar{\rho}_r d_r ) + \notag \\
&+(1+k) (d_{rr} \bar{\rho}^{\gamma-1} + (\gamma-1)\bar{\rho}^{\gamma-2} \bar{\rho}_{r} d_r ) - \frac{2}{r}((\gamma+k(\gamma-1)) \bar{\rho}^{\gamma-2} \bar{\rho}_r d + (1+k) \bar{\rho}^{\gamma-1} d_r ).
\end{align}
Then using the commutation rule for $\D_i L_0$ Lemma \ref{L:DIL0COMMUTATIONRESULT}
\begin{align}
&\mathcal{D}_i \left[ \frac{1}{\bar{\rho}} \partial_r  ( \bar{\rho}^\gamma d \tfrac{1}{r^2} \partial_r (r^2 H)) \right] = \mathcal{D}_i L_0 H =\mathcal{L}_i \mathcal{D}_i H + \sum_{j=0}^{i-1} q_{ij} \mathcal{D}_{i-j} H,
\end{align}
where
\begin{equation}
q_{ij}=\sum_{k=1}^{2+j} \frac{ \sum_{\ell =0}^k c_{i j k \ell} \partial_r^\ell d }{r^{2+j-k}},
\end{equation}
and $c_{i j k \ell}$ are bounded functions on $[0,1]$.
Now returning to (\ref{E:HEQNAPPLYDI}), we have
\begin{align}
&a^{3 \gamma-3} \mathcal{D}_i H_{\tau \tau} + a^{3 \gamma-4} a_{\tau} \mathcal{D}_i H_\tau -  \mathcal{D}_i H - \mathcal{D}_i \left[\frac{H^2}{r}\right] - \gamma \frac{\xi^4}{\J^{\gamma+1}} \mathcal{L}_i \mathcal{D}_i H \notag \\
&- \gamma \frac{\xi^4}{\J^{\gamma+1}} \sum_{j=0}^{i-1} q_{ij} \mathcal{D}_{i-j} H - \gamma \bar{\mathcal{D}}_{i-1} \left( \partial_r \left( \frac{\xi^4}{\J^{\gamma+1}} \right) L_0 H \right) -\gamma [\bar{\mathcal{D}}_{i-1},\frac{\xi^4}{\J^{\gamma+1}}] D_r L_0 H  \notag \\
&+ \mathcal{D}_i \left[ r \mathcal{R}_1[\tfrac{H}{r}] + r \mathcal{R}_2[\tfrac{H}{r}]  \right]  = 0. \label{E:RETURNTOHEQNAPPLYDI}
\end{align}
Let
\begin{equation}
C_i[H]:= \frac{\xi^4}{\J^{\gamma+1}} \sum_{j=0}^{i-1} q_{ij} \mathcal{D}_{i-j} H + [\bar{\mathcal{D}}_{i-1},\frac{\xi^4}{\J^{\gamma+1}}] D_r L_0 H.
\end{equation}
Then we write (\ref{E:RETURNTOHEQNAPPLYDI}) as follows 
\begin{align}
&a^{3 \gamma-3} \mathcal{D}_i H_{\tau \tau} + a^{3 \gamma-4} a_{\tau} \mathcal{D}_i H_\tau  - \gamma \frac{\xi^4}{\J^{\gamma+1}} \mathcal{L}_i \mathcal{D}_i H \notag \\
&= \mathcal{D}_i H +  \mathcal{D}_i \left[\frac{H^2}{r}\right] - \mathcal{D}_i \left[ r \mathcal{R}_1[\tfrac{H}{r}] + r \mathcal{R}_2[\tfrac{H}{r}]  \right] + \gamma C_i[H] + \gamma \bar{\mathcal{D}}_{i-1} \left( \partial_r \left( \frac{\xi^4}{\J^{\gamma+1}} \right) L_0 H \right). \label{E:PREESTIMATEHAPPLYDI}
\end{align}

\subsection{Time Weight Manipulation}\label{S:TWM}
For $\gamma > \frac53$ we need to eliminate our equivalent of the anti-damping effect encountered in \cite{1610.01666}. We use the strategy from \cite{shkoller2017global}, applied to our spherically symmetric positive density nonisentropic setting. First recall our $\gamma$ dependent exponents
\begin{equation}
d(\gamma)=\begin{cases}
3\gamma - 3 & \text{ if } \  1<\gamma\leq \frac53 \\
2 & \text{ if } \  \gamma>\frac53
\end{cases}; \quad b(\gamma)=d(\gamma)+3-3\gamma=\begin{cases}
0 & \text{ if } \  1<\gamma\leq \frac53 \\
5-3 \gamma & \text{ if } \  \gamma>\frac53
\end{cases}.
\end{equation} 
We then multiply (\ref{E:PREESTIMATEHAPPLYDI}) by $a^{b(\gamma)}$ to obtain
\begin{align}
&a^{d(\gamma)} \mathcal{D}_i H_{\tau \tau} + a^{d(\gamma)-1} a_{\tau} \mathcal{D}_i H_\tau  - \gamma a^{b(\gamma)} \frac{\xi^4}{\J^{\gamma+1}} \mathcal{L}_i \mathcal{D}_i H \notag \\
&=  a^{b(\gamma)} \mathcal{D}_i H +  a^{b(\gamma)} \mathcal{D}_i \left[\frac{H^2}{r}\right] - a^{b(\gamma)} \mathcal{D}_i \left[ r \mathcal{R}_1[\tfrac{H}{r}] + r \mathcal{R}_2[\tfrac{H}{r}]  \right] + \gamma a^{b(\gamma)} C_i[H] \notag \\
&\quad \quad \quad + \gamma a^{b(\gamma)} \bar{\mathcal{D}}_{i-1} \left( \partial_r \left( \frac{\xi^4}{\J^{\gamma+1}} \right) L_0 H \right). \label{E:PREESTIMATEHAPPLYDIPOSTGAMMADIFF}
\end{align}

\subsection{Energy Identity}
Multiply (\ref{E:PREESTIMATEHAPPLYDIPOSTGAMMADIFF}) by $r^2 d^i \mathcal{D}_i H_\tau$ and integrate in $r$ from $0$ to $1$
\begin{align}
& \int_0^1 a^{d(\gamma)} \mathcal{D}_i H_{\tau \tau} \mathcal{D}_i H_\tau d^i r^2 \, dr + \int_0^1 a^{d(\gamma)-1} a_{\tau} (\mathcal{D}_i H_{\tau})^2 d^i r^2 \, dr \notag \\
& -\gamma a^{b(\gamma)} \begin{cases}
\int_0^1 \frac{\xi^4}{\bar{\rho}^{1+k(\gamma-1)} \J^{\gamma+1}} \partial_r \left( \bar{\rho}^{\gamma+i(\gamma-1)} d^{1+i} D_r \mathcal{D}_i H  \right) r^2 \mathcal{D}_i H_\tau & \text{ if $i$ is even}\\
\int_0^1 \frac{\xi^4}{\bar{\rho}^{1+k(\gamma-1)} \J^{\gamma+1}} D_r \left( \bar{\rho}^{\gamma+i(\gamma-1)} d^{1+i} \partial_r \mathcal{D}_i H  \right) r^2 \mathcal{D}_i H_\tau & \text{ if $i$ is odd}
\end{cases} \notag \\
&=  \int_0^1 a^{b(\gamma)} \mathcal{D}_i H \mathcal{D}_i H_\tau r^2 d^i \, dr +  \int_0^1 a^{b(\gamma)} \mathcal{D}_i \left[\frac{H^2}{r} \right] \mathcal{D}_i H_\tau r^2 d^i \, dr \notag \\
&- \int_0^1 a^{b(\gamma)} \mathcal{D}_i \left[ r \mathcal{R}_1[\tfrac{H}{r}] + r \mathcal{R}_2[\tfrac{H}{r}]  \right] \mathcal{D}_i H_\tau r^2 d^i \, dr \notag \\
&+ \gamma \int_0^1 a^{b(\gamma)} C_i[H] r^2 d^i \mathcal{D}_i H_\tau \, dr + \gamma \int_0^1 a^{b(\gamma)} \bar{\mathcal{D}}_{i-1} \left( \partial_r \left( \frac{\xi^4}{\J^{\gamma+1}} \right) L_0 H \right) r^2 d^i \mathcal{D}_i \mathcal{H}_\tau \, dr. \label{E:POSTMULTIPLYDIHTAU}
\end{align}
We rewrite the third term on the left hand side of (\ref{E:POSTMULTIPLYDIHTAU}) and obtain
\begin{align}
& \int_0^1 a^{d(\gamma)} \mathcal{D}_i H_{\tau \tau} \mathcal{D}_i H_\tau d^i r^2 \, dr + \int_0^1 a^{d(\gamma)-1} a_{\tau} (\mathcal{D}_i H_{\tau})^2 d^i r^2 \, dr \notag \\
& + \gamma \int_0^1 a^{b(\gamma)} \frac{\xi^4}{\J^{\gamma+1}} \mathcal{D}_{i+1} H \mathcal{D}_{i+1} H_{\tau} r^2 \bar{\rho}^{\gamma-1} d^{i+1} \, dr \notag \\
&=  \int_0^1 a^{b(\gamma)} \mathcal{D}_i H \mathcal{D}_i H_\tau r^2 d^i \, dr + \int_0^1 a^{b(\gamma)} \mathcal{D}_i \left[\frac{H^2}{r} \right] \mathcal{D}_i H_\tau r^2 d^i \, dr \notag \\
&-\gamma \int_0^1 a^{b(\gamma)} \bar{\rho}^{\gamma+i(\gamma-1)} d^{1+i}  \mathcal{D}_{i+1} H \mathcal{D}_{i} H_\tau r^2 \partial_r \left( \frac{\xi^4}{\bar{\rho}^{1+i(\gamma-1)} \J^{\gamma+1}} \right) \, dr \notag \\
&- \int_0^1 a^{b(\gamma)} \mathcal{D}_i \left[ r \mathcal{R}_1[\tfrac{H}{r}] + r \mathcal{R}_2[\tfrac{H}{r}]\right] \mathcal{D}_i H_\tau r^2 d^i \, dr \notag \\
&+ \gamma \int_0^1 a^{b(\gamma)} C_i[H] r^2 d^i \mathcal{D}_i H_\tau \, dr + \gamma \int_0^1 a^{b(\gamma)} \bar{\mathcal{D}}_{i-1} \left( \partial_r \left( \frac{\xi^4}{\J^{\gamma+1}} \right) L_0 H \right) r^2 d^i \mathcal{D}_i \mathcal{H}_\tau \, dr. \label{E:POSTREWRITETHIRDTERM}
\end{align}
Writing (\ref{E:POSTREWRITETHIRDTERM}) using perfect time derivatives
\begin{align}
&\frac{d}{d \tau} \left( \frac{1}{2} a^{d(\gamma)} \int_0^1 ( \mathcal{D}_i H_\tau)^2 d^i r^2 \, dr + \frac{\gamma a^{b(\gamma)}}{2}  \int_0^1 \frac{\xi^4}{\J^{\gamma+1}} (\mathcal{D}_{i+1} H)^2 \bar{\rho}^{\gamma-1} d^{i+1} r^2 \, dr \right) \notag \\
&+\frac{2-d(\gamma)}{2} a^{d(\gamma)-1} a_{\tau} \int_0^1 (\mathcal{D}_i H_\tau)^2 d^i r^2 \, dr - \frac{\gamma b(\gamma)}{2} a^{b(\gamma)-1} a_{\tau} \int_0^1 \frac{\xi^4}{\J^{\gamma+1}} (\mathcal{D}_{i+1} H)^2 \bar{\rho}^{\gamma-1} d^{i+1} r^2 \, dr \notag \\
&= a^{b(\gamma)} \int_0^1 \mathcal{D}_i H \mathcal{D}_i H_\tau r^2 d^i \, dr +  a^{b(\gamma)} \int_0^1 \mathcal{D}_i \left[\frac{H^2}{r} \right] \mathcal{D}_i H_\tau r^2 d^i \, dr \notag \\ 
&+  \frac{\gamma a^{b(\gamma)}}{2}  \int_0^1 \partial_\tau \left(\frac{\xi^4}{\J^{\gamma+1}}\right) (\mathcal{D}_{i+1} H)^2 \bar{\rho}^{\gamma-1} d^{i+1} r^2 \, dr \notag \\
&-\gamma a^{b(\gamma)} \int_0^1 \bar{\rho}^{\gamma+i(\gamma-1)} d^{i+1}  \mathcal{D}_{i+1} H \mathcal{D}_{i} H_\tau r^2 \partial_r \left( \frac{\xi^4}{\bar{\rho}^{1+i(\gamma-1)} \J^{\gamma+1}} \right) \, dr \notag \\
&- \int_0^1 a^{b(\gamma)} \mathcal{D}_i \left[ r \mathcal{R}_1[\tfrac{H}{r}] + r \mathcal{R}_2[\tfrac{H}{r}]  \right] \mathcal{D}_i H_\tau r^2 d^i \, dr \notag \\
&+ \gamma \int_0^1 a^{b(\gamma)} C_i[H] r^2 d^i \mathcal{D}_i H_\tau \, dr + \gamma \int_0^1 a^{b(\gamma)} \bar{\mathcal{D}}_{i-1} \left( \partial_r \left( \frac{\xi^4}{\J^{\gamma+1}} \right) L_0 H \right) r^2 d^i \mathcal{D}_i \mathcal{H}_\tau \, dr. \label{E:ENERGYIDENTITY}
\end{align}
\begin{remark}
The second line of the left hand side of (\ref{E:ENERGYIDENTITY}) is nonnegative for both $\gamma \in (1,\tfrac{5}{3}]$ and $\gamma > \frac{5}{3}$ where we note $2-d(\gamma)=0$ for $\gamma > \frac{5}{3}$ and $b(\gamma)=0$ for $\gamma \in (1,\tfrac{5}{3}]$.
\end{remark}

\subsection{High-order Quantities from Energy Identity}
We now define our high order quantities. First from (\ref{E:ENERGYIDENTITY}), define
\begin{align}
\mathcal{E}_i&= \frac{1}{2} a^{d(\gamma)} \int_0^1 ( \mathcal{D}_i H_\tau)^2 d^i r^2 \, dr + \frac{\gamma a^{b(\gamma)}}{2}  \int_0^1 \frac{\xi^4}{\J^{\gamma+1}} (\mathcal{D}_{i+1} H)^2 \bar{\rho}^{\gamma-1} d^{i+1} r^2 \, dr \notag \\
\mathscr{D}_i&=\frac{2-d(\gamma)}{2} a^{d(\gamma)-1} a_{\tau} \int_0^1 (\mathcal{D}_i H_\tau)^2 d^i r^2 \, dr - \frac{\gamma b(\gamma)}{2} a^{b(\gamma)-1} a_{\tau} \int_0^1 \frac{\xi^4}{\J^{\gamma+1}} (\mathcal{D}_{i+1} H)^2 \bar{\rho}^{\gamma-1} d^{i+1} r^2 \, dr \notag \\
R_i &=  \int_0^1 a^{b(\gamma)} \mathcal{D}_i H \mathcal{D}_i H_\tau r^2 d^i \, dr +  \int_0^1 a^{b(\gamma)} \mathcal{D}_i \left[\frac{H^2}{r} \right] \mathcal{D}_i H_\tau r^2 d^i \, dr \notag \\ 
&+ \frac{\gamma}{2}  \int_0^1 a^{b(\gamma)} \partial_\tau \left(\frac{\xi^4}{\J^{\gamma+1}}\right) (\mathcal{D}_{i+1} H)^2 \bar{\rho}^{\gamma-1} d^{i+1} r^2 \, dr \notag \\
&-\gamma \int_0^1 a^{b(\gamma)} \bar{\rho}^{\gamma+i(\gamma-1)} d^{1+i}  \mathcal{D}_{i+1} H \mathcal{D}_{i} H_\tau r^2 \partial_r \left( \frac{\xi^4}{\bar{\rho}^{1+i(\gamma-1)} \J^{\gamma+1}} \right) \, dr \notag \\
&- \int_0^1 a^{b(\gamma)} \mathcal{D}_i \left[ r \mathcal{R}_1[\tfrac{H}{r}] + r \mathcal{R}_2[\tfrac{H}{r}] \right] \mathcal{D}_i H_\tau r^2 d^i \, dr \notag \\
&+ \gamma \int_0^1 a^{b(\gamma)} C_i[H] r^2 d^i \mathcal{D}_i H_\tau \, dr + \gamma \int_0^1 a^{b(\gamma)} \bar{\mathcal{D}}_{i-1} \left( \partial_r \left( \frac{\xi^4}{\J^{\gamma+1}} \right) L_0 H \right) r^2 d^i \mathcal{D}_i \mathcal{H}_\tau \, dr.
\end{align}
Then let
\begin{align}
\mathcal{E}^N (\tau) & = \sum_{i=0}^N \mathcal{E}_i(\tau) \\
\mathscr{D}^N (\tau) & = \sum_{i=0}^N \mathscr{D}_i(\tau)  \\
R_i &= Z_1^i + Z_2^i +Z_3^i +Z_4^i +Z_5^i +Z_6^i +Z_7^i.
\end{align}
where we will show in the zeroth order estimate $Z_6^0 =Z_7^0 = 0$. Therefore from (\ref{E:ENERGYIDENTITY}), we have
\begin{align}\label{E:ENERGYIDENTITYINTERMSOFHOQ}
\frac{d}{d \tau} ( \mathcal{E}^N (\tau) ) + \mathscr{D}^N (\tau) &= \sum_{i=0}^N R_i \notag \\
& =  \sum_{i=0}^N (Z_1^i + Z_2^i +Z_3^i +Z_4^i +Z_5^i +Z_6^i +Z_7^i).
\end{align}\label{E:ENERGYIDENTITYINTERMSOFHIGHORDERQUANTITIES}

\subsection{Main Energy Inequality}
To establish our central energy inequality, we primarily need to estimate the right hand of (\ref{E:ENERGYIDENTITYINTERMSOFHIGHORDERQUANTITIES}), that is, estimate $Z_j^i$ for all $j=1,2,3,4,5,6,7$ and all $0 \leq i \leq N$. 

Before this, we introduce a term only present for $\gamma>\frac{5}{3}$ which is similar to $\mathcal{E}_i$ but does not include top order quantities or time weights with negative powers, and will be controlled through our coercivity Lemma \ref{L:COERCIVITY}
\begin{equation}
\mathcal{C}_{i,\gamma}=\mathbf{1}_{ \gamma > 5/3 } \int_0^1 \frac{\gamma\,\xi^4}{2 \J^{\gamma+1}} (\mathcal{D}_{i+1} H)^2 \bar{\rho}^{\gamma-1} d^{i+1} r^2 \, dr,
\end{equation}
where $\mathbf{1}_{ \gamma > 5/3 }=\begin{cases}
1 & \text{ if } \gamma > \frac53 \\
0 & \text{ if } 1 < \gamma  \leq \frac53
\end{cases}$. Then let 
\begin{equation}
\mathcal{C}^{N-1} (\tau)  = \sum_{i=0}^{N-1} \mathcal{C}_{i,\gamma}(\tau).
\end{equation}
\begin{remark}
Our coercivity Lemma \ref{L:COERCIVITY} given in Appendix \ref{A:COERCIVITY} will let us control $\mathcal{C}^{N-1}$. However we cannot use Lemma \ref{L:COERCIVITY} to handle top order terms as that would require control of $N+1$ derivatives of $\partial_\tau H$. This is why we do not include top order terms in $\mathcal{C}^{N-1}$.
\end{remark}
Finally, prior to proving our main energy inequality, it is worth formally stating the equivalence of our high order norm $\mathcal{S}^N$ and high order energy functional $\mathcal{E}^N$.
\begin{lemma}\label{L:NORMENERGYGAMMAGREATER5OVER3} 
Let $(H, H_\tau):[0,1] \rightarrow \mathbb R \times \mathbb R$ be a unique local solution to (\ref{E:HEQN})-(\ref{E:THETAICGAMMALEQ5OVER3}) on $[0,T]$ for $T>0$ fixed with $\|H(0)\|_0^2 < \infty$ and assume $(H, H_\tau)$ satisfies the a priori assumptions (\ref{E:SNAPRIORI})-(\ref{E:DOUBLEPARTIALRTHETAAPRIORI}). Fix $N \geq 8$. Let $k \geq N$ in (\ref{E:PHIDEMAND}). Then there are constants $C_1,C_2>0$ so that
\begin{align}\label{E:NORMENERGYEQUIVALENCE}
C_1\mathcal{S}^N (\tau) \le \sup_{0\le\tau'\le\tau}\{\mathcal{E}^N(\tau')+\mathcal{C}^{N-1}(\tau')\} \le C_2(\mathcal{S}^N(\tau)+\mathcal{S}^N(0)).
\end{align} 
\end{lemma} 
\begin{proof}
Recall the definition $\mathcal{S}^N$ (\ref{E:SNNORM}). Then (\ref{E:NORMENERGYEQUIVALENCE}) is a straightforward application of the uniform boundedness of $\J$ (\ref{E:JAPRIORI}), $\xi$ (\ref{E:XIBOUNDEDABOVEANDBELOW}) and $\bar{\rho}$ (\ref{E:BARRHODEMAND}), in conjunction with Lemma \ref{L:COERCIVITY} to control terms without time weights with negative powers, which are included in $\mathcal{C}^{N-1}(\tau')$, by $\mathcal{S}^N(\tau)+\mathcal{S}^N(0)$.
\end{proof}
We are now ready to prove our central high order energy inequality which will be essential in the proof of our main result Theorem \ref{T:MAINTHEOREMGAMMALEQ5OVER3}. 
\begin{proposition}\label{P:ENERGYESTIMATEGAMMALEQ5OVER3}
Suppose $\gamma >1$. Let $(H, H_\tau):[0,1] \rightarrow \mathbb R \times \mathbb R$ be a unique local solution to (\ref{E:HEQN})-(\ref{E:THETAICGAMMALEQ5OVER3}) on $[0,T]$ for $T>0$ fixed with $\|H(0)\|_0^2 < \infty$ and assume $(H, H_\tau)$ satisfies the a priori assumptions (\ref{E:SNAPRIORI})-(\ref{E:DOUBLEPARTIALRTHETAAPRIORI}). Fix $N \geq 8$. Let $k \geq N$ in (\ref{E:PHIDEMAND}). Then for all $\tau \in [0,T]$, we have the following inequality for some $0<\kappa\ll 1$
\begin{equation}
\mathcal E^N(\tau) +\mathcal{C}^{N-1}(\tau)+\int_0^\tau \mathscr D^N(\tau')\,d\tau' \lesssim  \mathcal{S}^N(0) +  \| H(0) \|_0^2 + \kappa\mathcal{S}^N(\tau) + \int_0^\tau e^{-a_0\tau'} \mathcal{S}^N(\tau') d\tau'. \label{E:ENERGYMAINGAMMALEQ5OVER3}
\end{equation}
\end{proposition}
\begin{proof}
We integrate our energy identity written in terms of our high order quantities (\ref{E:ENERGYIDENTITYINTERMSOFHOQ}) from $0$ to $\tau$, and for $\gamma > \frac{5}{3}$ apply Lemma \ref{L:COERCIVITY}, to obtain the left hand side of (\ref{E:ENERGYMAINGAMMALEQ5OVER3}). Our goal is then to estimate $|Z_j^i|$ for all $j=1,2,3,4,5,6,7$ and all $0 \leq i \leq N$.

\underline{$j=1$}:
Fix $i \geq 1$. Suppose $i=\ell+1$ for $0 \leq \ell \leq N-1$. Then
\begin{align}
| Z^i_1 | & = |  a^{b(\gamma)} \int_0^1 \mathcal{D}_i  H \mathcal{D}_i H_\tau r^2 d^i \, dr | \notag \\
&\lesssim a^{b(\gamma)} ( \int_0^1 (\mathcal{D}_i  H)^2 r^2 d^i \, dr)^{1/2}  ( \int_0^1 (\mathcal{D}_i H_\tau)^2 r^2 d^i \, dr)^{1/2} \notag \\
& =a^{b(\gamma)} ( \int_0^1 (\mathcal{D}_{\ell+1} H)^2 r^2 d^{\ell+1} \, dr)^{1/2}  ( e^{d(\gamma) \tau} e^{-d(\gamma) \tau} \int_0^1 (\mathcal{D}_i H_\tau)^2 r^2 d^i \, dr)^{1/2} \notag \\
& \lesssim (\mathcal{S}^N)^{1/2} e^{-a_0 \tau} (\mathcal{S}^N)^{1/2}  = e^{-a_0 \tau} \mathcal{S}^N (\tau).
\end{align}
For $i=0$, we first need to derive the energy identity (\ref{E:ENERGYIDENTITY}) for $i=0$. Multiplying our equation for $H$ (\ref{E:HEQN}) by $a^{b(\gamma)}$ we have
\begin{equation}\label{E:HEQNPREZEROORDEREST}
a^{d(\gamma)} H_{\tau \tau} + a^{d(\gamma)-1} a_{\tau} H_\tau - \gamma a^{b(\gamma)}\frac{\xi^4}{\J^{\gamma+1}} L_0 H =  a^{b(\gamma)} H +  a^{b(\gamma)} \frac{H^2}{r} - a^{b(\gamma)} ( r \mathcal{R}_1 [\tfrac{H}{r}]+r \mathcal{R}_2 [\tfrac{H}{r}] + r \mathcal{R}_3 [\tfrac{H}{r}]).
\end{equation}
Multiply (\ref{E:HEQNPREZEROORDEREST}) by $r^2 H_\tau$ and integrate in $r$ from $0$ to $1$
\begin{align}
&\int_0^1 a^{d(\gamma)} H_{\tau \tau} H_{\tau} r^2 \, dr + \int_0^1 a^{d(\gamma)-1} a_\tau (H_\tau)^2 r^2 \, dr -\gamma \int_0^1 a^{b(\gamma)} \frac{\xi^4}{\bar{\rho} \J^{\gamma+1}} \partial_r ( \bar{\rho}^\gamma d D_r H) H_\tau r^2 \, dr \notag \\
&=\int_0^1 a^{b(\gamma)} H H_\tau r^2 \, dr +   \int_0^1 a^{b(\gamma)} \frac{H^2}{r} H_\tau r^2 \, dr - \int_0^1 a^{b(\gamma)} ( r \mathcal{R}_1 [\tfrac{H}{r}]+r \mathcal{R}_2 [\tfrac{H}{r}] + r \mathcal{R}_3 [\tfrac{H}{r}]) H_\tau r^2 \, dr. \label{E:POSTMULTIPLYZEROORDER}
\end{align}
Rewriting the third term of the left hand side of (\ref{E:POSTMULTIPLYZEROORDER}) we then have
\begin{align}
&\int_0^1 a^{d(\gamma)} H_{\tau \tau} H_{\tau} r^2 \, dr + \int_0^1 a^{d(\gamma)-1} a_\tau (H_\tau)^2 r^2 \, dr +\gamma \int_0^1 a^{b(\gamma)} \frac{\xi^4}{\bar{\rho} \J^{\gamma+1}} \mathcal{D}_1 H \mathcal{D}_1 H_\tau r^2 \bar{\rho}^{\gamma-1} d \, dr \notag \\
&=\int_0^1 a^{b(\gamma)} H H_\tau r^2 \, dr + \int_0^1 a^{b(\gamma)} \frac{H^2}{r} H_\tau r^2 \, dr - \gamma \int_0^1 a^{b(\gamma)} \bar{\rho}^\gamma d \mathcal{D}_1 H \partial_r \left( \frac{\xi^4}{\bar{\rho} \J^{\gamma+1}} \right) H_\tau r^2  \, dr \notag \\ 
&- \int_0^1 a^{b(\gamma)} ( r \mathcal{R}_1 [\tfrac{H}{r}]+r \mathcal{R}_2 [\tfrac{H}{r}] + r \mathcal{R}_3 [\tfrac{H}{r}]) H_\tau r^2 \, dr. \label{E:REWRITETHIRDTERMZEROORDER}
\end{align}
Writing (\ref{E:REWRITETHIRDTERMZEROORDER}) using perfect time derivatives
\begin{align}
&\frac{d}{d \tau} \left( \frac{1}{2} a^{d(\gamma)} \int_0^1 (H_\tau)^2 r^2 \, dr + \frac{\gamma}{2} a^{b(\gamma)} \int_0^1 \frac{\xi^4}{\J^{\gamma+1}} ( \mathcal{D}_1 H)^2 \bar{\rho}^{\gamma-1} d r^2 \, dr \right) \notag \\
&+ \frac{2-d(\gamma)}{2} a^{d(\gamma)-1} a_\tau \int_0^1 (H_\tau)^2 r^2 \, dr  - \frac{\gamma b(\gamma)}{2} a^{b(\gamma)-1} a_{\tau} \int_0^1 \frac{\xi^4}{\J^{\gamma+1}} ( \mathcal{D}_1 H)^2 \bar{\rho}^{\gamma-1} d r^2 \, dr \notag \\
&= a^{b(\gamma)} \int_0^1 H H_\tau r^2 \, dr + a^{b(\gamma)} \int_0^1 \frac{H^2}{r} H_\tau r^2 \, dr + \frac{\gamma}{2} a^{b(\gamma)} \int_0^1 \partial_\tau \left(\frac{\xi^4}{\J^{\gamma+1}}\right) ( \mathcal{D}_1 H)^2 \bar{\rho}^{\gamma-1} d r^2 \, dr  \notag \\
&- \gamma a^{b(\gamma)} \int_0^1 \bar{\rho}^\gamma d \mathcal{D}_1 H \partial_r \left( \frac{\xi^4}{\bar{\rho} \J^{\gamma+1}} \right) H_\tau r^2  \, dr  - a^{b(\gamma)} \int_0^1 ( r \mathcal{R}_1 [\tfrac{H}{r}]+r \mathcal{R}_2 [\tfrac{H}{r}] + r \mathcal{R}_3 [\tfrac{H}{r}]) H_\tau r^2 \, dr. \label{E:ZEROORDERENERGYIDENTITY}
\end{align}
Therefore in terms of our high order quantities at the zeroth order level, we have
\begin{align}
\frac{d}{d \tau} ( \mathcal{E}_0 (\tau) ) + \mathscr{D}_0 (\tau) &= R_0 \notag \\
& = Z_1^0 + Z_2^0 +Z_3^0 +Z_4^0 +Z_5^0 +Z_6^0 +Z_7^0,
\end{align}\label{E:ZEROORDERENERGYIDENTITYINTERMSOFHIGHORDERQUANTITIES}
where $Z_6^0=Z_7^0=0$. Returning to our $j=1$ estimate, that is our $|Z_1^i|$ estimate, for $i=0$ using Young's inequality with $0<\kappa\ll 1$ we have
\begin{align}
\left| \int_0^\tau  Z_0^1  d \tau' \right|  & = \left|  \int_0^\tau a^{b(\gamma)} \int_0^1 H H_\tau r^2 \, dr  d\tau' \right| \notag \\
& = \left|  \int_0^\tau \int_0^1 a^{b(\gamma)} a^{-\frac{d(\gamma)}{4}} H a^{\frac{d(\gamma)}{4}} H_\tau r^2 \, dr  d\tau' \right| \notag \\
& \lesssim \kappa \int_0^\tau \int_0^1 a^{-\frac{d(\gamma)}{2}} H^2 r^2 \, dr d \tau' + \int_0^\tau \int_0^1 a^{\frac{d(\gamma)}{2}} H_\tau^2 r^2 \, dr d\tau' \notag \\
& \lesssim \kappa \sup_{0 \leq \tau' \leq \tau} \{ \int_0^1 H^2 r^2 \, dr \} \int_0^\tau a^{-\frac{d(\gamma)}{2}} d \tau' + \int_0^\tau \int_0^1 a^{-\frac{d(\gamma)}{2}} a^{d(\gamma)} H_\tau^2 r^2 \, dr d\tau' \notag \\
& \lesssim \kappa \sup_{0 \leq \tau' \leq \tau} \{ \int_0^1 H^2 r^2 \, dr  \} + \int_0^\tau e^{-a_0 \tau'} \mathcal{S}^N (\tau') d \tau'.
\end{align}
For $\int_0^1 H^2 r^2 \, dr$ apply a coercivity estimate as follows
\begin{equation}
H=\int_0^\tau H_\tau d \tau' + H(0) = \int_0^\tau a^{-\frac{3 \gamma-3}{2}} a^{\frac{3 \gamma -3 }{2}} H_\tau d \tau' + H(0) \lesssim \sup_{0 \leq \tau' \leq \tau} a^{\frac{3 \gamma -3}{2}} H_\tau + H(0),
\end{equation}
and hence using Cauchy's inequality ($ab \lesssim a^2 + b^2,$ $a,b \in \mathbb{R}$)  we have
\begin{equation}
\int_0^1 (H)^2 r^2 \, dr \lesssim \mathcal{S}^N(\tau) + \| H(0) \|_0^2.
\end{equation}
Therefore 
\begin{align}
\left| \int_0^\tau Z_1^0 d \tau' \right| \lesssim \kappa \mathcal{S}^N(\tau) + \| H(0) \|_0^2 +  \int_0^\tau e^{-a_0 \tau'} \mathcal{S}^N(\tau').
\end{align}
\underline{$j=2$}: Fix $0 \leq i \leq N$. Then with $\psi \geq 0$ a smooth cut-off function such that $\psi=1$ on $[0,\tfrac{1}{2}]$, $\psi = 0$ on $[\tfrac{3}{4},1]$ and $\psi' \leq 0$,
\begin{align}
|Z^i_2| &= \left|  a^{b(\gamma)} \int_0^1 \mathcal{D}_i \left( \frac{H^2}{r} \right) \mathcal{D}_i H_\tau r^2 d^i \, dr \right| \notag \\
&=\left|  a^{b(\gamma)} \int_0^{\frac{3}{4}} \mathcal{D}_i \left( \frac{H^2}{r} \right) \mathcal{D}_i H_\tau r^2 d^i \psi \, dr +  a^{b(\gamma)} \int_{\frac{1}{2}}^1 \mathcal{D}_i \left( \frac{H^2}{r} \right) \mathcal{D}_i H_\tau r^2 d^i (1-\psi) \, dr \right| \notag \\
&\lesssim \left| \int_0^{\frac{3}{4}} \mathcal{D}_i \left( \frac{H^2}{r} \right) \mathcal{D}_i H_\tau r^2 d^i \psi \, dr \right| + \left| \int_{\frac{1}{2}}^1 \mathcal{D}_i \left( \frac{H^2}{r} \right) \mathcal{D}_i H_\tau r^2 d^i (1-\psi) \, dr \right|. \label{E:ZI2SPLIT}
\end{align}
Next recalling the definition of the vector field $\mathcal{P}_i$ given in Section \ref{S:NOTATION}, first note $\D_i \in \mathcal{P}_i$. Then by the product rule for $\PC_i$ 
Lemma \ref{L:product}
\begin{equation}
\D_i \left[ \frac{H^2}{r} \right] = \sum_{k=0}^i \sum_{A \in \PC_k \atop B \in \bar{\PC}_{i-k} } c_k^{\D_i A B} A [H] B \left[ \frac{H}{r} \right],
\end{equation}
for some real valued constants $c_k^{\D_i B C}$. Now note for $B \in \bar{\PC}_{i-k}$, $B \left[ \frac{H}{r} \right] = C[H]$ for $C \in \PC_{i+1-k}$. Then we can write in terms of low order and high order derivatives
\begin{equation}
\D_i \left[ \frac{H^2}{r} \right]=\sum_{k=0}^{i/2} \sum_{A \in \PC_k \atop C \in \PC_{i+1-k}} c_k^{\D_i A C} A [H] C[H].
\end{equation}
Hence applying Lemma \ref{L:control} to estimate $\PC_{i+1-k}$ using $\D_{i+1-k}$, our $L_\infty$ embedding for $\PC_k$ (\ref{E:LINFTY3}) and the fact that $1 \lesssim d \lesssim 1$ on $[0,\tfrac{3}{4}]$, we have for the left hand integral on the last line of (\ref{E:ZI2SPLIT})
\begin{align}
&\left| \int_0^{\frac{3}{4}} \mathcal{D}_i \left( \frac{H}{r} \right) \mathcal{D}_i H_\tau r^2  d^i \psi \, dr \right| \lesssim \sum_{k=0}^{i/2} \sum_{A \in \PC_k \atop C \in \PC_{i+1-k}} \|A(H)\|_{L^\infty} \left| \int_0^{\frac{3}{4}} C[ H ] \mathcal{D}_i H_\tau r^2 d^i \psi \, dr \right| \notag \\
& \lesssim (\mathcal{S}^N(\tau) + \| H(0) \|^2  )\sum_{k=0}^{i/2} \sum_{C \in \PC_{i+1-k}} \left( \int_0^{\frac{3}{4}} ( C [H] )^2 r^2 d^i \psi^2 \, dr \right)^{1/2} \left( \int_0^{\frac{3}{4}} \psi^2 (\mathcal{D}_i H_\tau)^2 r^2 d^i \, dr \right)^{1/2} \notag \\
& \lesssim \sum_{k=0}^{i/2} \sum_{C \in \PC_{i+1-k}} \left( \int_0^{\frac{3}{4}} ( C[ H ] )^2 r^2 \psi^2 \, dr \right)^{1/2} \left( \int_0^1 (\mathcal{D}_i H_\tau)^2 r^2 d^i \, dr \right)^{1/2} \notag \\
& \lesssim \sum_{k=0}^{i/2} \left( \int_0^{\frac{3}{4}} ( \mathcal{D}_{i+1-k} H )^2 r^2 \psi^2 \, dr \right)^{1/2} e^{-a_0 \tau} (\mathcal{S}^N(\tau))^{\frac{1}{2}} \notag \\
& \lesssim \sum_{k=0}^{i/2} \left( \int_0^{\frac{3}{4}} ( \mathcal{D}_{i+1-k} H )^2 r^2 \psi^2 d^{i+1} \, dr \right)^{1/2} e^{-a_0 \tau} (\mathcal{S}^N(\tau))^{\frac{1}{2}} \notag \\
& \lesssim  (\mathcal{S}^N(\tau))^{\frac{1}{2}} e^{-a_0 \tau} (\mathcal{S}^N(\tau))^{\frac{1}{2}} \lesssim e^{-a_0 \tau} \mathcal{S}^N(\tau).
\end{align}
For the second integral on the right hand side of (\ref{E:ZI2SPLIT}), also use
\begin{equation}
\PC_k u = \sum_{\ell=0}^k c_\ell (r) \D_\ell (H),
\end{equation}
where $c_\ell$ are smooth functions of $r$ on $[\tfrac{1}{4},1]$ in this case to obtain
\begin{align}
& \left| \int_{\frac{1}{2}}^1 \mathcal{D}_i \left( \frac{H^2}{r} \right) \mathcal{D}_i H_\tau r^2 d^i (1-\psi) \, dr \right|   \lesssim \sum_{k=0}^{i/2} \sum_{A \in \PC_k \atop C \in \PC_{i+1-k}}  \|A(H)\|_{L^\infty} \left( \int_{\tfrac{1}{2}}^1 C[H] \mathcal{D}_i H_\tau r^2 d^i (1-\psi) \, dr \right)\notag \\
& \lesssim (\mathcal{S}^N(\tau) + \| H(0) \|^2  )  \sum_{k=0}^{i/2} \sum_{\ell=0}^{i+1-k} \int_{\tfrac{1}{2}}^1 \D_{\ell} H \mathcal{D}_i H_\tau r^2 d^i (1-\psi) \, dr \notag \\
& \lesssim \sum_{k=0}^{i/2} \sum_{\ell=0}^{i+1-k} \kappa ( \int_{\frac{1}{2}}^1 a^{-\frac{d(\gamma)}{2}} (\mathcal{D}_\ell H)^2 r^2 d^i \, dr ) + \int_{\frac{1}{2}}^1 a^{\frac{d(\gamma)}{2}} (\mathcal{D}_i H_\tau)^2 r^2 d^i \, dr  \notag \\
&\lesssim  a^{-\frac{d(\gamma)}{2}} \kappa \mathcal{S}^N(\tau) + a^{-\frac{d(\gamma)}{2}} \|H(0)\|^2  +   e^{-a_0 \tau} \mathcal{S}^N(\tau).
\end{align}
Therefore combining the above analysis
\begin{equation}
\int_0^\tau Z^i_2 d \tau' \lesssim \kappa \mathcal{S}^N(\tau) + \| H(0) \|^2_0 + \int_0^\tau e^{-a_0 \tau'} \mathcal{S}^N(\tau') d \tau'.
\end{equation}
\underline{$j=3$}:
For $Z_3^i=\frac{\gamma}{2} \int_0^1 \partial_\tau \left( \frac{ \xi^4}{\J^{\gamma+1}} \right) ( \mathcal{D}_{i+1} H)^2 \bar{\rho}^{\gamma-1} d^{i+1} r^2 dr$, first compute:
\begin{equation}
\partial_\tau (\xi^4 \J^{-(\gamma+1)})  = 4 \xi^3 (\partial_\tau \xi) \J^{-(\gamma+1)} - (\gamma+1) \J^{-(\gamma+2)} (\partial_\tau \J) \xi^4.
\end{equation}
Now using $\xi=\uptheta+1$
\begin{align}
\partial_\tau \xi &= \partial_\tau \uptheta = \partial_\tau ( \tfrac{1}{r} H ) = \tfrac{1}{r} H_\tau, \notag \\ 
\partial_\tau \J &= \partial_\tau (\xi^3 + \xi^2 \xi_r r ), \notag \\ 
\partial_\tau (\xi_r) &= \partial_r (\uptheta_\tau) = \partial_r \left(\frac{H_\tau}{r} \right), \notag \\
r \partial_\tau (\xi_r) &= r \partial_r (\frac{H_\tau}{r}) = r(-\frac{1}{r^2} H_\tau + \frac{1}{r} \partial_r H_\tau ) = \partial_r H_\tau - \frac{1}{r} H_\tau  = \D_1 H_\tau -\frac{3}{r} H_\tau. 
\end{align}
Hence
\begin{align}
&\left| Z^i_3 \right| = \frac{\gamma}{2} \left| \int_0^1 \left [ (4 \xi^3 \J^{-(\gamma+1)}) \frac{H_\tau}{r} -(\gamma+1) \xi^4 \J^{-(\gamma+2)}(3 \xi^2 \frac{H_\tau}{r} + 2 \xi H_\tau + \D_1 H_\tau  \right. \right.  \notag \\
&\left. \left.  \qquad \qquad \qquad \qquad  -3 \frac{H_\tau}{r}) \right] (\D_{i+1} H)^2 \bar{\rho}^{\gamma-1} d^{i+1} r^2 \, dr \right| \notag \\
& \lesssim \left| \int_0^1 \frac{H_\tau}{r} ( \D_{i+1} H)^2 d^{i+1} r^2 \, dr \right| + \left| \int_0^1 H_\tau ( \D_{i+1} H)^2 d^{i+1} r^2 \, dr \right| + \left| \int_0^1 ( \D_1 H_\tau ) ( \D_{i+1} H)^2 d^{i+1} r^2 \, dr \right| \notag \\
& \lesssim ( \| \frac{H_\tau}{r} \|_{\infty} + \|H_\tau \|_{\infty} + \| \mathcal{D}_1 H_\tau \|_{\infty}) \int_0^1 (\D_{i+1} H)^2 \, d^{i+1} r^2 \, dr. \label{E:ZI3FIRSTBOUND}
\end{align}
Applying the $L^\infty$ embedding (\ref{E:LINFTY2})
\begin{align}
\| \frac{H_\tau}{r} \|_{\infty} &\lesssim \left( \sum_{k=2}^3 \int_0^{\frac{3}{4}} (\mathcal{D}_k H_\tau)^2 r^2 \, dr + \sum_{k=2}^3 \int_{\frac{1}{4}}^1 d^2 (\mathcal{D}_k H_\tau)^2 \, dr  \right)^{\frac{1}{2}} \notag \\
& \lesssim \left( \sum_{k=2}^3 \int_0^{\frac{3}{4}} (\mathcal{D}_k H_\tau)^2 r^2 d^k \, dr + \sum_{k=0}^2 \int_{\frac{1}{4}}^1 r^2 d^2 (\mathcal{D}_k H_\tau)^2 \, dr  \right)^{\frac{1}{2}} \notag \\
& \lesssim \left( \sum_{k=2}^3 \int_0^{1} (\mathcal{D}_k H_\tau)^2 r^2 d^k \, dr + \sum_{k=0}^2 \int_{0}^1 r^2 d^2 (\mathcal{D}_k H_\tau)^2 \, dr  \right)^{\frac{1}{2}} \notag \\
& \lesssim e^{-a_0 \tau} (\mathcal{S}^N)^{1/2}.
\end{align}
Applying the $L^\infty$ embedding (\ref{E:LINFTY1A}) 
\begin{align}
\| H_\tau \|_{\infty} &\lesssim \left( \sum_{k=2}^3 \int_0^{\frac{3}{4}} (\mathcal{D}_k H_\tau)^2 r^2 \, dr + \sum_{k=0}^2 \int_{\frac{1}{4}}^1 d^2 (\mathcal{D}_k H_\tau)^2 \, dr  \right)^{\frac{1}{2}} \notag \\
& \lesssim \left( \sum_{k=2}^3 \int_0^{\frac{3}{4}} (\mathcal{D}_k H_\tau)^2 r^2 d^k \, dr + \sum_{k=0}^2 \int_{\frac{1}{4}}^1 r^2 d^2 (\mathcal{D}_k H_\tau)^2 \, dr  \right)^{\frac{1}{2}} \notag \\
& \lesssim \left( \sum_{k=2}^3 \int_0^{1} (\mathcal{D}_k H_\tau)^2 r^2 d^k \, dr + \sum_{k=0}^2 \int_{0}^1 r^2 d^2 (\mathcal{D}_k H_\tau)^2 \, dr  \right)^{\frac{1}{2}} \notag \\
& \lesssim e^{-a_0 \tau} (\mathcal{S}^N)^{1/2}.
\end{align}
Applying the $L^\infty$ embedding  (\ref{E:LINFTY1B}) and using that $\bar{\mathcal{D}}_k \mathcal{D}_1 = \mathcal{D}_{k+1}$ 
\begin{align}
\| \mathcal{D}_1 H_\tau \|_{\infty} &\lesssim \left( \sum_{k=1}^2 \int_0^{\frac{3}{4}} (\bar{\mathcal{D}}_k \mathcal{D}_1 H_\tau)^2 r^2 \, dr + \sum_{k=0}^3 \int_{\frac{1}{4}}^1 d^4 (\bar{\mathcal{D}}_k \mathcal{D}_1 H_\tau)^2 \, dr  \right)^{\frac{1}{2}} \notag \\
& \lesssim \left( \sum_{k=1}^2 \int_0^{\frac{3}{4}} (\mathcal{D}_{k+1} H_\tau)^2 r^2 d^{k+1} \, dr + \sum_{k=0}^3 \int_{\frac{1}{4}}^1 r^2 d^{k+1} (\mathcal{D}_{k+1} H_\tau)^2 \, dr  \right)^{\frac{1}{2}} \notag \\
& \lesssim \left( \sum_{k=1}^2 \int_0^{1} (\mathcal{D}_{k+1} H_\tau)^2 r^2 d^{k+1} \, dr + \sum_{k=0}^3 \int_{0}^1 r^2 d^{k+1} (\mathcal{D}_k H_\tau)^2 \, dr  \right)^{\frac{1}{2}} \notag \\
& \lesssim e^{-a_0 \tau} (\mathcal{S}^N)^{1/2}.
\end{align}
Then since $\int_0^1 (\D_{i+1} H)^2 \, d^{i+1} r^2 \, dr \lesssim \mathcal{S}^N$ from (\ref{E:ZI3FIRSTBOUND}) we have 
\begin{equation}
| Z^i_3 | \lesssim e^{-a_0 \tau} (\mathcal{S}^N)^{\frac{3}{2}} \lesssim e^{-a_0 \tau} \mathcal{S}^N.
\end{equation}
\underline{$j=4$}:
For $Z^i_4=-\gamma \int_0^1 a^{b(\gamma)} \bar{\rho}^{\gamma+i(\gamma-1)} d^{1+i}  \mathcal{D}_{i+1} H \mathcal{D}_{i} H_\tau r^2 \partial_r \left( \frac{\xi^4}{\bar{\rho}^{1+i(\gamma-1)} \J^{\gamma+1}} \right) \, dr$ first note
\begin{align}
| \partial_r (\xi^4 \J^{-(\gamma+1)} \bar{\rho}^{-(1+i(\gamma-1))} | &= | 4 \xi^3 \xi_r \J^{-(\gamma+1)} \bar{\rho}^{-(1+i(\gamma-1))} - (\gamma+1) \J^{-(\gamma+2)} \J_r \bar{\rho}^{-(1+i(\gamma-1)} \notag \\
&- (1+i(\gamma-1)) \bar{\rho}^{-(2+i(\gamma-1))} \bar{\rho}_r \xi^4 \J^{-(\gamma+1)} | \notag \\
&\lesssim 1.
\end{align}
Then
\begin{align}
|Z^i_4| & \lesssim \left| \int_0^1 d^{1+i} \mathcal{D}_{i+1} H \mathcal{D}_i H_\tau r^2 \, dr \right| \notag \\
& \lesssim ( \int_0^1 d^{1+i} (\mathcal{D}_{i+1} H)^2 r^2 \, dr )^{1/2} (\int_0^1 d^i (\mathcal{D}_i H_\tau)^2 r^2 \, dr)^{1/2} \notag \\
& \lesssim e^{-a_0 \tau} (\mathcal{S}^N)^{1/2} (\mathcal{S}^N)^{1/2} = e^{-a_0 \tau} \mathcal{S}^N.
\end{align}
\underline{$j=5$}:
Recall
\begin{equation}
Z^i_5=- \int_0^1 a^{b(\gamma)} \mathcal{D}_i \left[ r \mathcal{R}_1[\tfrac{H}{r}] + r \mathcal{R}_2[\tfrac{H}{r}] \right] \mathcal{D}_i H_\tau r^2 d^i \, dr.
\end{equation}
We start by considering 
\begin{equation}
- \int_0^1 a^{b(\gamma)} \mathcal{D}_i \left[r \mathcal{R}_2[\tfrac{H}{r}] \right] \mathcal{D}_i H_\tau r^2 d^i \, dr = - \int_0^1 \bar{\mathcal{D}}_{i-1} \left[ D_r [ r \mathcal{R}_2[\tfrac{H}{r} ] ]  \right] \mathcal{D}_i H_\tau r^2 d^i \, dr.
\end{equation}
since $\mathcal{R}_2$ requires more care. We combine this term with the \underline{$j=7$}, $Z^i_7$, term which is as follows
\begin{equation}
\gamma \int_0^1 a^{b(\gamma)} \bar{\mathcal{D}}_{i-1} \left( \partial_r \left( \frac{\xi^4}{\J^{\gamma+1}} \right) L_0 H \right) r^2 d^i \mathcal{D}_i \mathcal{H}_\tau \, dr
\end{equation}
to obtain
\begin{equation}
a^{b(\gamma)} \int_0^1 \bar{\mathcal{D}}_{i-1} \left( \gamma \partial_r (\frac{\xi^4}{\J^{\gamma+1}}) L_0 H - D_r [ r \mathcal{R}_2 [ \frac{H}{r} ] ] \right) D_i H_\tau r^2 d^i \, dr : = Z^i_{\mathcal{R}_2 5 7} \text{ for } i \geq 1.
\end{equation}
We note for $i=0$, we have $Z^0_{\mathcal{R}_2 5 7}:=-\int_0^1 r \mathcal{R}_3 [ \frac{H}{r} ] H_\tau r^2 \, dr$. We examine the structure of $Z^i_{\mathcal{R}_2 5 7}$ for $i \geq 1$. Note 
\begin{align}
&\gamma \partial_r (\frac{\xi^4}{\J^{\gamma+1}}) L_0 H - D_r [ r \mathcal{R}_2 [ \frac{H}{r} ] ] =-\left[ \gamma \partial_r ( \J^{-\gamma-1} r (D_r H) \xi^4 + \partial_r ( r \mathcal{R}_2 [ \frac{H}{r} ] ) \right] \notag \\
&+ \gamma \partial_r ( \xi^4 ) \J^{-\gamma-1} L_0 H + \gamma \xi^4 \bar{\rho}^{\gamma-1} d (\partial_r D_r H) \partial_r (\J^{-\gamma-1}) - \frac{2}{r} ( r \mathcal{R}_2 [ \frac{H}{r} ] ). \label{E:SQUAREBRACKETSCONCERNINGTERM}
\end{align}
Next note
\begin{align}
\mathcal{R}_2[\uptheta] &= \mathcal{R}_2^a [\uptheta] + \mathcal{R}^b_2 [\uptheta] \text{ where } \notag \\
\mathcal{R}_2^a [\uptheta] & := -\xi^2 ( (\J^{-\gamma} -1) + \gamma (\J-1) + \gamma ( \J^{-\gamma-1} - 1) \xi^2 ( r \partial_r \uptheta + 3 \uptheta) ), \notag \\
\mathcal{R}_2^b [\uptheta] &:=- \gamma \xi^2 (3 \uptheta^2 + 2 \uptheta^3 ).
\end{align}
Now
\begin{align}
\partial_r ( r \mathcal{R}^a_2 [ \uptheta ] ) &= - \partial_r [ r \xi^2  ((\J^{-\gamma} - 1) + \gamma (\J - 1) + \gamma (\J^{-\gamma-1}-1)\xi^2 ( r \partial_r \uptheta + 3 \uptheta) ) ] \notag \\
&=-\partial_r (r \xi^2) ( (\J^{-\gamma}-1) + \gamma (\J-1) + \gamma( \J^{-\gamma-1} -1) \xi^2 (r \partial_r \uptheta + 3 \uptheta)) \notag \\
&-r \xi^2 \partial_r \left[ (\J^{-\gamma-1} + \gamma (\J - 1) + \gamma (\J^{-\gamma-1}-1) \xi^2 (r \partial_r \uptheta + 3 \uptheta) \right].
\end{align}
Using $\partial_r ( \J^{-\gamma}-1) = -\gamma \J^{-\gamma-1} \partial_r \J$ we have
\begin{align}
&\partial_r \left[ (\J^{-\gamma-1} + \gamma (\J - 1) + \gamma (\J^{-\gamma-1}-1) \xi^2 (r \partial_r \uptheta + 3 \uptheta) \right] \notag \\
&= -\gamma \J^{-\gamma-1} \partial_r \J + \gamma \partial_r \J - (\gamma(\gamma+1) \J^{-\gamma-2} \partial_r \J) \xi^2 (r \partial_r \uptheta + 3 \uptheta) \notag \\
&+ \gamma (\J^{-\gamma-1}-1) \partial_r (\xi^2 (r \partial_r \uptheta + 3 \uptheta)) \notag \\
&=-\gamma(\gamma+1)\J^{-\gamma-2} (\partial_r \J) \xi^2 (r \partial_r \uptheta + 3 \uptheta) + \gamma(\J^{-\gamma-1}-1) \partial_r (\xi^2 (r \partial_r \uptheta+3\uptheta) ) \notag \\
&+ (\gamma-\gamma \J^{-\gamma-1}) \partial_r \J \notag \\
&=-\gamma(\gamma+1)\J^{-\gamma-2} (\partial_r \J) \xi^2 (r \partial_r \uptheta + 3 \uptheta) - \gamma (\J^{-\gamma-1} - 1) (\partial_r \J - \xi^2 (r \partial_r^2 \uptheta + 4 \partial_r \uptheta)) \notag \\
&+ 2 \gamma (\J^{-\gamma-1}-1) \xi \xi_r (r \partial_r \uptheta + 3 \uptheta).
\end{align}
Thus using $r \partial_r (\frac{H}{r}) + 3 \frac{H}{r} =D_r H$ we can write
\begin{equation}
\partial_r ( r \mathcal{R}^a_2 [ \uptheta ] )  = r \xi^2 \gamma(\gamma+1) \J^{-\gamma-2} (\partial_r \J ) \xi^2 D_r H + \mathcal{R}_3[\uptheta],
\end{equation}
where we define
\begin{align}
&\mathcal{R}_3[\uptheta]=\mathcal{R}_3[\frac{H}{r}]:=\notag \\
&\quad r \xi^2 \gamma (\J^{-\gamma-1} - 1)(\partial_r \J - \xi^2 ( r \partial_r^2 \uptheta + 4 \partial_r \uptheta)) -2 r \gamma \xi^2 (\J^{-\gamma-1} -1)\xi \xi_r (r \partial_r \uptheta + 3 \uptheta)  \notag \\
&\quad - \partial_r (r \xi^2) [(\J^{-\gamma}-1) + \gamma (\J-1) + \gamma (\J^{-\gamma-1}-1) \xi^2 (r \partial_r \uptheta + 3 \uptheta)].
\end{align}
Returning to (\ref{E:SQUAREBRACKETSCONCERNINGTERM}), noting $\partial_r ( \J^{-\gamma}-1) = -\gamma \J^{-\gamma-1} \partial_r \J$, we can then cancel the unfavorable with respect to weight $d$ term $\gamma(\gamma+1)\J^{-\gamma-2} (\partial_r \J) r D_r H \xi^4$ and obtain
\begin{align}
&\gamma \partial_r (\frac{\xi^4}{\J^{\gamma+1}}) L_0 H - D_r [ r \mathcal{R}_2 [ \frac{H}{r} ] ] = -\gamma r \partial_r (\xi^4) \J^{-\gamma-1} D_r H + \gamma \partial_r (\xi^4) \bar{\rho}^{\gamma-1} d (\partial_r D_r H) \J^{-\gamma-1} \notag \\
&+\gamma \xi^4 \bar{\rho}^{\gamma-1} d (\partial_r D_r H) \partial_r \J^{-\gamma-1}  - \frac{2}{r} ( r \mathcal{R}_2 [ \frac{H}{r} ])  - \mathcal{R}_3 [ \frac{H}{r} ] - \partial_r [ r \mathcal{R}^b_2 [\frac{H}{r} ] ] \notag \\
&=-\gamma r \partial_r (\xi^4) \J^{-\gamma-1} D_r H + \gamma  \bar{\rho}^{\gamma-1} d \partial_r ( \frac{\xi^4}{\J^{\gamma+1}}) (\partial_r D_r H)  - \mathcal{R}_3 [ \frac{H}{r} ] - \frac{2}{r} ( r \mathcal{R}_2 [ \frac{H}{r} ])  - \partial_r [ r \mathcal{R}^b_2 [\frac{H}{r} ] ].
\end{align}
In terms of derivative count with respect to weight $d$, the only potentially concerning component above is $r \xi^2 \gamma (\J^{-\gamma-1} - 1)(\partial_r \J - \xi^2 ( r \partial_r^2 \uptheta + 4 \partial_r \uptheta))$ in $\mathcal{R}_3 [ \frac{H}{r} ] = \mathcal{R}_3 [ \uptheta ]$. Noting $\J=\xi^3+\xi^2 \xi_r r$ we rewrite this term in a favorable form using the following identity
\begin{equation}
\partial_r \J - \xi^2 ( r \partial_r^2 \uptheta + 4 \partial_r \uptheta ) = 2 \xi (\partial_r \uptheta)^2 r.
\end{equation}
Therefore we have written $Z^i_{\mathcal{R}_2 5 7}$ for $i \geq 1$ in a desirable form.

We now estimate $Z^i_{\mathcal{R}_2 5 7}$ for $i=0$ and then for $i \geq 1$. 

For $i=0$, note 
\begin{align}
Z^0_{\mathcal{R}_2 5 7} &=- a^{b(\gamma)} \int_0^1 r \mathcal{R}_2[\uptheta] H_{\tau} r^2 \, dr \notag \\
&=- a^{b(\gamma)} \int_0^1 r (- \xi^2 ( (\J^{\gamma-1} - 1) + \gamma (\J-1)) H_{\tau} r^2 \, dr  \notag \\
&- \int_0^1 a^{b(\gamma)} r (- \gamma \xi^2 ( (\J^{-\gamma-1}-1) \xi^2 (r \partial_r \uptheta + 3 \uptheta) + a^{b(\gamma)} (3 \uptheta^2 + 2 \uptheta^3)) H_\tau r^2 \, dr \notag \\
&=(i)+(ii).
\end{align}
For $(i)$, we use Taylor series to write
\begin{equation}\label{E:TAYLORSERIESJ}
\J^{-\gamma}-1+\gamma(\J-1)=\gamma(\gamma+1)\left( \int_0^1 (1-s)(1+s(\J-1))^{-\gamma-2} \, ds \right)(\J-1)^2.
\end{equation}
Now
\begin{align}
&(\J-1)^2=((1+\uptheta)^2(1+\uptheta+\uptheta_r r)-1)^2 =( (1+\tfrac{H}{r})^2(1+H_r)-1)^2 \notag \\
&=((1+\tfrac{2H}{r}+\tfrac{H^2}{r^2})(1+H_r)-1)^2 = (1+H_r+\tfrac{2}{r}H+\tfrac{2}{r}H H_r + \tfrac{H^2}{r^2}+\tfrac{H^2 H_r}{r} - 1)^2 \notag \\
&=(D_r H + \tfrac{2}{r} H (D_r H - 2 \tfrac{H}{r}) + \tfrac{H^2}{r^2} + \tfrac{H^2}{r}(D_r H - \tfrac{2}{r} H))^2 \notag \\
&=(D_r H + D_r H \tfrac{2}{r} H + D_r H \tfrac{H^2}{r} - \tfrac{3}{r^2} H^2 - \tfrac{2 H^3}{r^2} )^2 \notag \\
&=(D_r H)^2 + \tfrac{4 H (D_r H)^2}{r} + \tfrac{2 H^2 (D_r H)^2}{r} - \tfrac{6 H^2 D_r H}{r} + \tfrac{4 H^2 (D_r H)^2}{r^2} - \tfrac{4 H^3 (D_r H)^2}{r^2} + \notag \\
&+\tfrac{H^4 (D_r H)^2}{r^2} - \tfrac{12 H^3 (D_r H)}{r^3} - \tfrac{14 H^4 D_r H}{r^3} - \tfrac{4 H^5 D_r H}{r^3} + \tfrac{9 H^4}{r^4} + \tfrac{12 H^5}{r^4} + \tfrac{4 H^6}{r^4}. \label{E:JMINUS1EXPAND}
\end{align}
First using (\ref{E:TAYLORSERIESJ}) and the boundedness of  $\xi^2 \gamma(\gamma+1)\left( \int_0^1 (1-s)(1+s(\J-1))^{-\gamma-2} \, ds \right)$, we have for integral $(i)$
\begin{equation}
|(i)| \lesssim \left| \int_0^1 r^3 (\J-1)^2 H_\tau dr \right|.
\end{equation}
As can be seen from (\ref{E:JMINUS1EXPAND}) many contributions from $(\J-1)^2$ are similar. Therefore we give the key estimates below and remark similar arguments will hold for the other terms.

For the $(D_r H)^2$ term in (\ref{E:JMINUS1EXPAND}), apply the $L^\infty$ embedding (\ref{E:LINFTY1A}) 
\begin{equation}
\| \int_0^1 r^3 (D_r H)^2 H_\tau \, dr | \lesssim \| H_\tau \|_{\infty} \mathcal{S}^N \lesssim e^{-a_0 \tau} \mathcal{S}^N.
\end{equation}
For the $\tfrac{2 H^2 (D_r H)^2}{r}$ term in (\ref{E:JMINUS1EXPAND}), use the $L^\infty$ embedding  (\ref{E:LINFTY1A}) and a similar argument to that used for $Z_0^1$
\begin{align}
\left| \int_0^1 r^2 H^2 (D_r H)^2 H_\tau \, dr \right| & \lesssim \| \mathcal{D}_1 H \|^2_{\infty} \left( \kappa a^{-\frac{d(\gamma)}{2}} \int_0^1 r^2 H^4 \, dr + a^{ d\frac{d(\gamma)}{2}} \int_0^1 r^2 (H_\tau)^2 \, dr \right) \notag \\
&\lesssim \mathcal{S}^N (\kappa a^{-\frac{d(\gamma)}{2}} \| H \|_\infty^2 \int_0^1 r^2 H^2 \, dr + e^{-a_0 \tau} \mathcal{S}^N) \notag \\
&\lesssim \kappa a^{-\frac{d(\gamma)}{2}} \mathcal{S}^N + a^{\frac{d(\gamma)}{2}} \|H(0)\|_0^2+e^{-a_0 \tau} \mathcal{S}^N,
\end{align}
and therefore 
\begin{equation}
\left| \int_0^\tau \int_0^1 r^3 H^2 (D_r H)^2 H_\tau \, dr d \tau' \right| \lesssim \kappa \mathcal{S}^N  + \|H(0)\|_0^2 + \int_0^\tau e^{-a_0 \tau'} \mathcal{S}^N d \tau'.
\end{equation}

For integral $(ii)$ 
\begin{align}
| (ii) | &\lesssim \left| \int_0^1 r^3 \mathcal{D}_1 H H_\tau \, dr \right| + \left| \int_0^1 r H^2 H_\tau \, dr \right| + \left| \int_0^1 H^3 H_\tau \, dr \right| \notag \\
& \lesssim (\mathcal{S}^N)^{1/2} e^{-a_0 \tau} (\mathcal{S}^N)^{1/2} + \kappa a^{-\frac{d(\gamma)}{2}}  ( \| \tfrac{H}{r} \|_{\infty}^2+ \| \tfrac{H}{r} \|_{\infty}^4)  \int_0^1 H^2 r^2 \, dr + e^{-a_0 \tau} \mathcal{S}^N.
\end{align}
Therefore 
\begin{equation}
\left| \int_0^\tau (ii) d \tau' \right| \lesssim e^{-a_0 \tau} \mathcal{S}^N + \kappa \mathcal{S}^N + \| H(0) \|^2_0 +\int_0^\tau e^{-a_0 \tau'} \mathcal{S}^N d \tau'.
\end{equation}
Thus we have
\begin{equation}
\int_0^\tau | Z^0_{\mathcal{R}_2 5 7} | d \tau' \lesssim \kappa \mathcal{S}^N + \kappa \| H(0) \|^2_0 + \int_0^\tau e^{-a_0 \tau'} \mathcal{S}^N d \tau'.
\end{equation}
Recalling
\begin{align}
\mathcal{R}_3[\uptheta]&= r \xi^2 \gamma (\J^{-\gamma-1} - 1)(2 \xi (\uptheta_r)^2 r) -2 r \gamma \xi^2 (\J^{-\gamma-1} -1)\xi \xi_r (r \partial_r \uptheta + 3 \uptheta)  \notag \\
&- \partial_r (r \xi^2) [(\J^{-\gamma}-1) + \gamma (\J-1) + \gamma (\J^{-\gamma-1}-1) \xi^2 (r \partial_r \uptheta + 3 \uptheta)],
\end{align}
note for $i \geq 1$
\begin{align}
&Z^i_{\mathcal{R}_2 5 7} =a^{b(\gamma)} \int_0^1 \bar{\mathcal{D}}_{i-1}\left[ -\gamma r \partial_r (\xi^4) \J^{-\gamma-1} D_r H + \gamma \bar{\rho}^{\gamma-1} d \partial_r ( \tfrac{\xi^4}{\J^{\gamma+1}} ) (\partial_r D_r H) - \mathcal{R}_3[\uptheta] \right. \notag \\
&\left.-\tfrac{2}{r}(r \mathcal{R}_2[\tfrac{H}{4}]) -\partial_r [ r \mathcal{R}^b_2[\tfrac{H}{r}]]\right] \mathcal{D}_i H_\tau r^2 d^i \, dr \notag \\
&=a^{b(\gamma)} \int_0^1 \left( \bar{\mathcal{D}}_{i-1} [-\gamma r \partial_r (\xi^4) \J^{\gamma-1} D_r H ] + \bar{\mathcal{D}}_{i-1} [\gamma \bar{\rho}^{\gamma-1} d \partial_r ( \tfrac{\xi^4}{\J^{\gamma+1}} ) (\partial_r D_r H)] - \bar{\D}_{i-1} [ \mathcal{R}_3 [\uptheta]] \right. \notag \\
&\left. -\bar{\D}_{i-1} (2 \mathcal{R}_2 [\tfrac{H}{r}]) - \bar{\D}_{i-1} [\partial_r (r \mathcal{R}^b_2 (\tfrac{H}{r})] \right) \D_i H_\tau \, r^2 d^i \, dr \notag \\
&=\int_0^1 a^{b(\gamma)} \bar{\D}_{i-1} [ - \gamma r \partial_r (\xi^4) \J^{\gamma-1} D_r H ] \D_i H_\tau d^i r^2 \, dr \notag \\
&+\int_0^1 a^{b(\gamma)} \bar{\D}_{i-1} [ \gamma \bar{\rho}^{\gamma-1} d \partial_r ( \tfrac{\xi^4}{\J^{\gamma+1}} ) (\partial_r D_r H) ]  \bar{\D}_i d^i r^2 \, dr -\int_0^1 a^{b(\gamma)} \bar{\D}_{i-1} [ \R_3 [ \uptheta ] ] \D H_\tau d^i r^2 \, dr \notag \\
& - \int_0^1 a^{b(\gamma)} \bar{\D}_{i-1} ( 2 \R_2 [ \tfrac{H}{r} ] ) \bar{\D}_i d^i r^2 \, dr - \int_0^1 a^{b(\gamma)} \bar{\D}_{i-1} (\partial_r (r \R_2^b [ \tfrac{H}{r} ] ) ) \D_i H_\tau d^i r^2 \, dr \notag \\
&=(a)+(b)+(c)+(d)+(e).
\end{align}
For $(a)$ first rewrite
\begin{equation}
r \partial_r (\xi^4) = r 4 \xi^3 \xi_r = 4 \xi^3 ( r \partial_r ( \tfrac{H}{r} ) )  = 4 \xi^3 (D_r H - 3 \tfrac{H}{r}) = 4 \xi^3 D_r H - 12 \xi^3 \tfrac{H}{r}.
\end{equation}
By the product rule for $\PC_i,\bar{\PC}_i$ Lemma \ref{L:product} applied twice
\begin{equation}
\bar{\D}_{i-1} [ \xi^3 \J^{\gamma-1} D_r H D_r H ] = \sum_{ \substack{A_{1} \in \PC_{\ell+1}, A_2 \in \PC_{\ell_2+1}, \\ A_3 \in \bar{\PC}_{\ell_3}, A_4 \in \bar{\PC}_{\ell_4} \\ \ell_1 + \ell_2 + \ell_3 + \ell_4 = i-1 }} c^{A_1 A_2 A_3 A_4}_{\ell_1 \ell_2 \ell_3 \ell_4} A_1 (H) A_2(H) A_3 (\xi^3) A_4 (\J^{\gamma-1}).
\end{equation}
Now by the chain rule for $\bar{\PC}_i$ Lemma \ref{L:CHAINRULE}, for $A_3 \in \bar{\PC}_{\ell_3}$,
\begin{equation}
A_3 (\xi^3) = \sum_{k=1}^{\ell_3} \xi^{3-k} \sum_{i_1 + ... + i_k = \ell_3 \atop (A_3)_j \in \bar{\PC}_{i_j} } c_{k,i_1,...,i_k} \prod_{j=1}^k (A_3)_{j} \xi.
\end{equation}
Next note for $(A_3)_j \in \bar{\PC}_{i_j}$
\begin{equation}
(A_3)_j (\xi) = (A_3)_{j-1} \partial_r (\xi) = (A_3)_{j-1} \partial_r (\tfrac{H}{r}) = (A_3)_{j+1} (H)
\end{equation}
for some $(A_3)_{j-1} \in \PC_{i_j-1}$, $(A_3)_{j+1} \in \PC_{i_j+1}$. Similarly apply Lemma \ref{L:CHAINRULE} to obtain for $A_4 \in \bar{\PC}_{\ell_4}$
\begin{equation}
A_4 ( \J^{\gamma-1} ) = \sum_{k=1}^{\ell_4} \J^{\gamma-1-k} \sum_{i_1 + ... + i_k = \ell_4 \atop (A_4)_j \in \bar{\PC}_{i_j} } c_{k,i_1,...,i_k} \prod_{j=1}^k (A_4)_{j} \J.
\end{equation}
Now for $(A_4)_j \in \bar{\PC}_{i_j}$ 
\begin{equation}
(A_4)_j (\J)  = (A_4)_j ( \xi^3 + \xi^2 \xi_r r )= (A_4)_j ( \xi^3 ) + (A_4)_j (\xi^2 \xi_r r)
\end{equation}
We already know how to handle $(A_4)_j ( \xi^3 )$ and we can apply Lemma \ref{L:product} and Lemma \ref{L:CHAINRULE} to handle $(A_4)_j (\xi^2 \xi_r r)$ noting that
\begin{equation}
\xi_r r = r \partial_r ( \tfrac{H}{r} ) = D_r H -3 \tfrac{H}{r},
\end{equation}
and for some $\bar{A}_{i_j} \in \bar{\PC}_{i_j}$, $A_{i_j - 1} \in \PC_{i_j-1}$, $A_{i_j + 1} \in \PC_{i_j+1}$,
\begin{equation}
\bar{A}_{i_j} ( D_r H ) = A_{i_j - 1} \partial_r (D_r H) = A_{i_j + 1} (H).
\end{equation}
An analogous computation applies to $\bar{\D}_{i-1}  [ \xi^3 \J^{\gamma-1} \tfrac{H}{r} D_r H ] $ since $\tfrac{1}{r} \in \{D_r,\tfrac{1}{r} \}$, see the definition of $\PC_k$ for $k\geq 1$ (\ref{E:PEVEN})-(\ref{E:PODD}). 
Therefore for non-empty finite sets
\begin{align}
L&:=\{ \ell \in \mathbb{N}_{>0} : \lceil \tfrac{i}{4} \rceil \leq \ell \leq i \} \neq \emptyset \text { such that }  |L| < \infty \\
K(\ell) &:= \{ k \in  \mathbb{N}_0 : k \leq \lceil \tfrac{i-1}{2} \rceil +1 \} \text{ such that } 3 \leq |K(\ell)| < \infty.
\end{align}
we have 
\begin{equation}
\bar{\D}_{i-1} [ -\gamma r \partial_r ( \xi^4 ) \J^{\gamma-1} D_r H  ] = \sum_{\ell \in L \atop A_{\ell} \in \PC_\ell}  [ c_{\ell} A_{\ell} (H)  \prod_{k \in K(\ell) \atop A_k \in \PC_k } A_k (H) ],
\end{equation}
where $c_\ell$ are bounded coefficients on $[0,1]$ from above and below, $1 \lesssim |c_\ell| \lesssim 1$.
We can now estimate $(a)$
\begin{align}
|(a)|&=\left| \int_0^1 \bar{\D}_{i-1} [-\gamma r \partial_r (\xi^4) \J^{\gamma-1} D_r H ] \D_i H_\tau d^i r^2  \, dr \right|  \notag \\
&\lesssim \sum_{\ell \in L \atop A_{\ell} \in \PC_\ell} \prod_{k \in K(\ell) \atop A_k \in \PC_k } \| A_k (H) \|_{L^\infty} \left| \int_0^1 A_\ell (H) \D_i H_\tau d^i r^2 \, dr \right| \notag \\
& \lesssim \sum_{\ell \in L \atop A_{\ell} \in \PC_\ell} (\SC^N + \| H(0) \|^2_0)(\kappa a^{-\frac{d(\gamma)}{2}} \int_0^1 ( A_\ell (H) )^2 r^2 d^i \, dr +   a^{ d\frac{d(\gamma)}{2}}\int_0^1 (\D_i H_\tau)^2 r^2 d^i \, dr) \notag \\
& \lesssim e^{-a_0 \tau} \SC^N + \kappa (\SC^N + \| H(0) \|^2_0) (\SC^N + \| H(0) \|^2_0 ) \notag \\
& \lesssim e^{-a_0 \tau} \SC^N + \kappa \SC^N + \| H(0) \|^2_0.
\end{align}
Therefore
\begin{equation}
\int_0^\tau |(a)| d \tau' \lesssim \kappa \SC^N + \| H(0) \|^2_0 + \int_0^\tau e^{-a_0 \tau'} \SC^N d \tau'.
\end{equation}
For $(b)$ we have 
\begin{align}
\bar{\D}_{i-1} [ \bar{\rho}^{\gamma-1} d \partial_r ( \tfrac{\xi^4}{\J^{\gamma+1}} ) \partial_r D_r H ] &= \sum_{A_L \in \PC_L \atop 2 \leq L \leq i+1} c_L A_L(H) d \, ( \prod_{ A_k(L) \in \PC_{k(L)} \atop k(L) \leq \lfloor \tfrac{i-1}{2} \rfloor + 2 } A_{k(L)} (H) )   \notag \\
&+ \sum_{A_\ell \in \PC_\ell \atop 2 \leq \ell \leq i} c_\ell A_\ell(H) \prod_{ A_k(\ell) \in \PC_{k(\ell)} \atop k(\ell) \leq \lfloor \tfrac{i-1}{2} \rfloor + 2 } A_{k(\ell)} (H).
\end{align}

Then by a similar argument to $(a)$, we obtain the same bound.

For $(c)$, there are no troubling terms that are beyond similar methods to the methods of $(a)$ and $(b)$, so we obtain the same bound for $(c)$.

For $(d)$, an analogous argument gives the same bound.

For $(e)$, an analogous argument gives the same bound.

Therefore we have estimated $Z^i_{\R_2 5 7}$ for $i \geq 1$. Next we estimate
\begin{equation}
- \int_0^1 a^{b(\gamma)} \D_i [ r \R_1 [ \tfrac{H}{r} ] ] \D_i H_\tau r^2 d^i \, dr  := Z^i_{\R_1 5}.
\end{equation}
Note $\D_i = \bar{\D}_{i-1} D_r = \bar{\D}_{i-1} ( \partial_r + \tfrac{2}{r} ).$ Then there are no troubling terms that are beyond similar methods to the methods of $(a)$ and $(b)$ from $Z^i_{
\R_2 5 7}$ for $i \geq 1$, and methods similar to previous estimates, and therefore we obtain the same bound as for $(a)$ and $(b)$ from $Z^i_{\R_2 5 7}$ for $Z^i_{\R_1 5}.$

$\underline{j=6:}$ Note using the definition of $q_{ij}$ (\ref{E:QIJDEFN}) and the commutator identity Lemma \ref{L:COMMUTATORIDENTITY}
\begin{align}
C_i[H]  &= \frac{\xi^4}{\J^{\gamma+1}} q_{ij} \D_{i-j} H + [\bar{\D}_{i-1},\tfrac{\xi^4}{\J^{\gamma+1}} ] D_r L_0 H \notag \\
&= \frac{\xi^4}{\J^{\gamma+1}} \sum_{j=0}^{i-1} \left( \left( \sum_{k=1}^{2+j} \frac{\sum_{\ell=0}^k c \, \partial_r^\ell d}{r^{2+j-k}} \right) \D_{i-j} H \right) + (i-1) \partial_r ( \tfrac{\xi^4}{\J^{\gamma+1}} ) \bar{\D}_{i-2} (D_r L_0 H) \notag \\
&+ \sum_{2 \leq k \leq i \atop A \in \bar{\PC}_k, B \in \bar{\PC}_{i-k} } c \, A(\tfrac{\xi^4}{\J^{\gamma+1}}) B(D_r L_0 H).
\end{align}
Then note 
\begin{align}
& D_r L_0 H = D_r ( \tfrac{1}{\bar{\rho}} \partial_r (\bar{\rho}^\gamma d D_r H ) ) = D_r ( \tfrac{1}{\bar{\rho}} ( \partial_r ( \bar{\rho}^{\gamma} d ) D_r H + \bar{\rho}^\gamma d (\partial_r D_r H) ) \notag \\
&=D_r ( \tfrac{ \partial_r (\bar{\rho}^{\gamma} d )}{\bar{\rho}}  D_r H + \bar{\rho}^{\gamma-1} d ( \partial_r D_r H ) ) \notag \\
& =  D_r ( \tfrac{ \partial_r (\bar{\rho}^{\gamma} d) }{\bar{\rho}} )  D_r H + \tfrac{ \partial_r (\bar{\rho}^\gamma d)}{\bar{\rho}} ( \partial_r D_r H ) + \partial_r (\bar{\rho}^{\gamma-1} d ) \partial_r D_r H + \bar{\rho}^{\gamma-1} d ( D_r \partial_r D_r H). 
\end{align}
Therefore we can write $Z^i_6$ in a form to which can use analogous methods to the methods of $(a)$ and $(b)$ from $Z^i_{\R_2 5 7}$ for $i \geq 1$, and methods similar to previous estimates, and therefore we obtain the same bound as for $(a)$ and $(b)$ from $Z^i_{\R_2 5 7}$ for $Z^i_6$.

This completes the energy estimate.
\end{proof}

\section{Proof of the Main Theorem}\label{S:PROOFOFMAINTHEOREM}
Applying the Local Well-Posedness Theorem \ref{T:LWP} we have that on some time interval $[0,T], T>0$ there exists a unique solution to (\ref{E:HEQN}). By Proposition \ref{P:ENERGYESTIMATEGAMMALEQ5OVER3} with $\kappa>0$ chosen small enough and using the norm-modified energy equivalence Lemma \ref{L:NORMENERGYGAMMAGREATER5OVER3}, we obtain that there are universal constants $c_1,c_2,c_3,c_4\geq 1$ such that for any $0 \leq \tau^* \leq \tau \leq T$
\begin{equation}\label{E:ENERGYBOUND}
\mathcal{S}^N(\tau;\tau^*) \leq c_1 \mathcal{S}^N(\tau^*)+c_2 \lambda +c_3 \mathcal{S}^N(0) + c_4 \int_{\tau^*}^\tau e^{-a_0 \tau'} \mathcal{S}^N (\tau';\tau^*) d \tau'.
\end{equation}
Above $\mathcal{S}^N(\tau;\tau^*)$ denotes $\mathcal{S}^N$ with $\sup_{\tau^*\le \tau' \le\tau}$ instead of $\sup_{0\le \tau' \le\tau}$. Applying a standard well-posedness estimate, we have that the time of existence $T$ is inversely proportional to the size of the initial data, that is: $T \sim (\mathcal{S}^N(0))^{-1}$. Choose $\varepsilon>0$ so small that the time of existence $T$ satisfies 
\begin{equation}\label{E:LOCALTIME}
e^{-a_0 T/4} \leq \frac{\kappa a_0}{c_4}, \ \ \sup_{\tau \leq T}\mathcal{S}^N(\tau) \leq c\left(\mathcal{S}^N(0)+\lambda \right)
\end{equation}
where $c>0$ is a universal constant from the local well-posedness theory. Then let
\begin{equation}
C_\ast= 3(c_1 c + c_2 + c_3).
\end{equation}
Now define
\begin{equation}\label{E:CURLTDEFINITION}
\mathcal T : = \sup_{\tau\geq0}\{\text{ solution to (\ref{E:HEQN}) exists on $[0,\tau)$ and} \ \mathcal{S}^N(\tau)\leq  C_\ast \left(\mathcal{S}^N(0)+\lambda\right)\} . 
\end{equation}
Notice that $\mathcal T\geq T$ since $c \leq C_\ast$. Setting $\tau^*=\frac T2$ in (\ref{E:ENERGYBOUND}) for any $\tau\in[\frac T2,\mathcal T]$ we have
\begin{equation}
\mathcal{S}^N(\tau; \frac T2) \leq c_1 \mathcal{S}^N(\frac T2)+c_2 \lambda + c_3 \mathcal{S}^N(0) +  c_4 \int_{\frac T2}^\tau e^{-\frac{a_0}{2}\tau'}\mathcal{S}^N(\tau';\frac{T}{2})\,d\tau'.
\end{equation}
Therefore, applying (\ref{E:LOCALTIME}) we deduce that for any $\tau\in[\frac T2,\mathcal T]$
\begin{align}
\mathcal{S}^N(\tau; \frac T2) & \leq  c_1\mathcal{S}^N(\frac T2)+c_2 \lambda+c_3 \mathcal{S}^N(0)   + \frac{c_4}{a_0}e^{-a_0 T/4}\mathcal{S}^N(\tau;\frac{T}{2}) \notag \\
& \le c_1\mathcal{S}^N(\frac T2)+c_2 \lambda+c_3 \mathcal{S}^N(0) +\kappa\mathcal{S}^N(\tau;\frac{T}{2}). \label{E:CONT1}
\end{align}
By (\ref{E:LOCALTIME}), $\mathcal{S}^N(\frac T2)\leq c\left(\mathcal{S}^N(0)+\lambda \right)$ and so from (\ref{E:CONT1})
\begin{equation}
\mathcal{S}^N(\tau; \frac T2) \leq c_1 c \left( \mathcal{S}^N(0)+ \lambda \right) + c_2 \lambda+c_3\mathcal{S}^N(0) +\kappa\mathcal{S}^N(\tau;\frac{T}{2}).
\end{equation}
For $\kappa$ sufficiently small this gives
\begin{equation}
\mathcal{S}^N(\tau; \frac T2) \leq 2(c_1 c + c_2 + c_3) \left( \mathcal{S}^N(0) + \lambda \right) < C_\ast \left(\mathcal{S}^N(0) + \lambda\right),
\end{equation}
and thus
\begin{equation}\label{E:SNCSTARBOUND}
\mathcal{S}^N(\tau)< C_\ast \left(\mathcal{S}^N(0)+ \lambda\right).
\end{equation}
It is now straightforward to verify the a priori bounds (\ref{E:SNAPRIORI})-(\ref{E:DOUBLEPARTIALRTHETAAPRIORI}) are in fact improved. For example, by the fundamental theorem of calculus, with $A \in \PC_2$
\begin{equation}
| \partial_r \left(\frac{H}{r} \right) | = | \int_0^\tau A H_\tau | \lesssim \int_0^\tau e^{-a_0\tau'} \mathcal{S}^N(\tau')\,d\tau' \lesssim \varepsilon <\frac16, \ \ \tau\in[0,\mathcal T) 
\end{equation}
for $\varepsilon>0$ small enough. Analogous arguments apply to the remaining a priori assumptions. Recalling the definition of $\mathcal T$ \ref{E:CURLTDEFINITION} and by the continuity of the map $\tau\mapsto \mathcal{S}^N(\tau')$, we conclude that $\mathcal T = \infty$ and the solution to (\ref{E:HEQN}) exists globally-in-time. Moreover the global bound (\ref{E:GLOBALBOUND}) follows.

\section*{Acknowledgments}
C. Rickard acknowledges the support of the NSF grant DMS-1608494 and the NSF grant DMS-1613135.

\appendix
\renewcommand{\theequation}{\Alph{section}.\arabic{equation}}
\setcounter{theorem}{0}\renewcommand{\theorem}{\Alph{section}.\??arabic{prop}}

\section{Differential Operators}
We give a series of useful results concerning our differential operators $\D_i$ and vector fields $\PC_i$. First we have the product rule for $\D_i$ which is straightforward to prove using induction.
\begin{lemma}\label{L:PRODUCTRULEDI}
For any $i \geq  1$ the following holds
\begin{equation}
\D_i (fg) = (\D_i f) g + \bar{\D}_{i-1} (f \partial_r g) + \bar{\D}_{i-1} (g D_r f) + g (\bar{\D}_{i-1} D_r f).
\end{equation}
\end{lemma}
Next we have a commutation rule for $D_i L_0$ which also follows by induction, in a similar fashion to Lemma B.1 \cite{guo2018continued}
\begin{lemma}\label{L:DIL0COMMUTATIONRESULT}
For any $i \in \mathbb{Z}_{>0}$
\begin{equation}\label{E:CFORMULA}
\mathcal{D}_i L_0 X = \mathcal{L}_i \mathcal{D}_i X + \sum_{j=0}^{i-1} q_{ij} \mathcal{D}_{i-j} X
\end{equation}
where
\begin{equation}\label{E:QIJDEFN}
q_{ij}=\sum_{k=1}^{2+j} \frac{ \sum_{\ell =0}^k c_{i j k \ell} \partial_r^\ell d }{r^{2+j-k}},
\end{equation}
and $c_{i j k \ell}$ are bounded functions on $[0,1]$.
\end{lemma}
Next for the vector fields $\PC_i$ and $\bar{\PC}_i$, we give the product rule and chain rule from Lemma A.4 \cite{guo2018continued}  and Lemma A.5 \cite{guo2018continued} respectively.
\begin{lemma}\label{L:product}  
Let $i\in\mathbb N$ be given. 
\begin{enumerate}
\item[(a)]
For any $A\in \mathcal P_i$ the following identity holds:
\begin{equation}\label{E:product}
A \left( f g \right)= \sum_{k=0}^i \sum_{ B \in \mathcal P_k \atop C\in \bar{\mathcal P}_{i-k}} c^{ABC}_{k} \ Bf\, Cg,
\end{equation}
for some real-valued constants $c^{ABC}_{k}$.
\item[(b)]
For any $A\in \bar{\mathcal P}_i$ the following identity holds:
\begin{equation}\label{E:PRODUCTRULEDBAR}
A \left( f g \right)= \sum_{k=0}^i \sum_{ B \in \bar{\mathcal P}_k \atop C\in \bar{\mathcal P}_{i-k}} {\bar c}^{ABC}_{k} \ Bf\, Cg,
\end{equation}
for some real-valued constants ${\bar c}^{ABC}_{k}$.
\end{enumerate}
\end{lemma}
\begin{lemma}\label{L:CHAINRULE}
Let $a\in\mathbb R, i\in\mathbb N$ be given and fix a vectorfield $W\in\bar{\mathcal P}_i$.  Then for any sufficiently smooth $f$ the following identity holds
\begin{align}\label{E:CR}
W(f^a) = \sum_{k=1}^i f^{a-k} \sum_{i_1+\dots i_k=i \atop W_j\in \bar{\mathcal P}_{i_j}} c_{k,i_1,\dots,i_k} \prod_{j=1}^k W_j f
\end{align}
for some real constants $c_{k,i_1,\dots,i_k}$.
\end{lemma}
Next we give a useful commutator identity from Lemma B.2 \cite{guo2018continued}
\begin{lemma}\label{L:COMMUTATORIDENTITY}
For any $i \in \mathbb{Z}_{>0}$
\begin{equation}
[\bar{\D}_i, e]X  = 
 i \partial_r e {\bar\D}_{i-1} X
+ \sum_{2 \leq k\leq i \atop A\in \bar{\mathcal P}_{k}, B\in \bar{\mathcal P}_{i-k}} c \, (A e) (B X).
\end{equation}
\end{lemma}
Finally we give an important estimate from Lemma A.3 \cite{guo2018continued} which lets us control $\PC_i$ using $\D_i$.
\begin{lemma}\label{L:control} Suppose $\D_i X$ is bounded in $L^2([0,\frac34], r^2d r)$. Then we have the following estimate: 
\begin{equation}
\sum_{ \mathfrak D_i   \in \mathcal P_i}\int_0^\frac{3}{4} | \mathfrak D_i X|^2  r^2\psi^2 d r \lesssim \int_0^\frac{3}{4} |\D_i X|^2  r^2\psi^2 d r 
\end{equation}
where $\psi \geq 0$ is a smooth cutoff function satisfying $\psi=1$ on $[0,\frac12]$, $\psi=0$ on $[\frac34,1]$, and $\psi'\leq 0$. 
\end{lemma}

\section{Hardy-Sobolev Embedding}
From Lemma C.2 \cite{guo2018continued} we have the following weighted $L^\infty$ embedding
\begin{lemma}
For any smooth $u:[0,1] \rightarrow \mathbb{R}$ and any $m \in \mathbb{Z}_{>0}$, we have
\begin{align}
\| u \|_{\infty}^2 & \lesssim \sum_{k=1}^2 \int_0^{\frac{3}{4}} ( \mathcal{D}_k u)^2 r^2 \, dr + \sum_{k=0}^{m+1} \int_{\frac{1}{4}}^1 d^{2m} (\mathcal{D}_k u)^2 \, dr \label{E:LINFTY1A} \\
\|u \|_{\infty}^2 & \lesssim \sum_{k=1}^2  \int_0^{\frac{3}{4}} ( \bar{\mathcal{D}}_k u)^2 r^2 \, dr + \sum_{k=0}^{m+1} \int_{\frac{1}{4}}^1 d^{2m} (\bar{\mathcal{D}}_k u)^2 \, dr \label{E:LINFTY1B} \\
\| \frac{u}{r} \|_{\infty}^2 & \lesssim \sum_{k=1}^2  \int_0^{\frac{3}{4}} ( \mathcal{D}_k u)^2 r^2 \, dr + \sum_{k=0}^{m+1} \int_{\frac{1}{4}}^1 d^{2m} (\mathcal{D}_k u)^2 \, dr. \label{E:LINFTY2}
\end{align}
\end{lemma}
As a corollary we have the following embedding result specifically involving the vector field $\PC_k$
\begin{cor}
For any smooth $u:[0,1] \rightarrow \mathbb{R}$ and any $m \in \mathbb{Z}_{>0}$, we have for $k \in \mathbb{Z}_{>0}$
\begin{equation}\label{E:LINFTY3}
\| \PC_k u \|_{\infty}^2 \lesssim \sum_{j=1}^2 \int_0^{\tfrac{3}{4}} ( \PC_{j+k} u)^2 r^2 \, dr + \sum_{j=0}^{m+1} \int_{\tfrac{1}{4}}^1 d^{2m} (\PC_{j+k} u)^2 \, dr.
\end{equation}
\end{cor}
\begin{proof} 
First note
\begin{equation}
\PC_k u = \sum_{\ell=0}^k c_\ell (r) \D_\ell (H)
\end{equation}
where $c_\ell$ are smooth functions of $r$ on $[\tfrac{1}{4},1]$. Next note
\begin{align}
\D_j (\PC_k u) &= \PC_{j+k} u \text{ for } k \text{ even,} \\
\bar{\D}_j (\PC_k u) &= \PC_{j+k} u \text{ for } k \text{ odd.}
\end{align}
Then (\ref{E:LINFTY3}) follows from (\ref{E:LINFTY1A})-(\ref{E:LINFTY1B}).
\end{proof}

\section{Time Based Inequalities}
We have a simple but crucial result concerning $a(\tau)$ and the related quantities $a_1$ and $a_0$.
\begin{lemma}\label{L:USEFULTAULEMMAGAMMALEQ5OVER3}
Assume $\gamma > 1$. Fix an affine motion $a(t)$ from the set $\mathscr{S}$ under consideration.
Let
\begin{equation}\label{E:MU1MU0DEFNGAMMALEQ5OVER3}
a_1:=\lim_{\tau \rightarrow \infty} \frac{a_\tau(\tau)}{a(\tau)}, \quad a_0:=\frac{d(\gamma)}{2}a_1,
\end{equation}        
where $d(\gamma)=\begin{cases}
3\gamma - 3 & \text{ if } \  1<\gamma\leq \frac53 \\
2 & \text{ if } \  \gamma>\frac53
\end{cases}$. Then
\begin{align}
0<a_0&=a_0(\gamma)\le a_1, \label{E:MU0MU1INEQGAMMALEQ5OVER3}\\
e^{a_1\tau} \lesssim & \ a(\tau)  \lesssim  e^{a_1\tau}, \ \ \tau\ge0. \label{E:EXPMU1MUINEQGAMMALEQ5OVER3}
\end{align}    
\end{lemma} 
\begin{proof}
The result (\ref{E:MU0MU1INEQGAMMALEQ5OVER3}) immediately follows from the definition of $a_0$. For (\ref{E:EXPMU1MUINEQGAMMALEQ5OVER3}), when $1 < \gamma \leq \frac{5}{3}$, we directly apply Lemma A.1 \cite{1610.01666} to obtain the result by considering the special case of the full 3D problem where $A(t)=\text{diag}(a(t),a(t),a(t))$. For $\gamma > \frac{5}{3}$, using (\ref{E:AFFREQ1}) and similarly considering $A(t)=\text{diag}(a(t),a(t),a(t))$, we can apply Lemma 6 \cite{MR3634025} to express $a(t)$ in the form
\begin{equation}
a(t)=a_0+ta_1+m(t),
\end{equation}
such that $a_0,a_1$ are time-independent and $m(t)$ satisfies the bounds
\begin{equation}
|m(t)|=o_{t \rightarrow \infty} (1+t), \, |\partial_t m(t) | \lesssim (1+t)^{3-3\gamma}.
\end{equation}
We also recall $a(t) \sim 1+t$. Then (\ref{E:EXPMU1MUINEQGAMMALEQ5OVER3}) follows from Lemma A.1 \cite{1610.01666} in this case also.
\end{proof}

\section{Local Well-Posedness}\label{A:LWP}
Here we sketch the proof of the local well posedness Theorem \ref{T:LWP} for our system. We first recall Theorem \ref{T:LWP}.
\begin{manualtheorem}{2.1}
Suppose $\gamma > 1$. Fix $N \geq 8$. Let $k \geq N$ in (\ref{E:PHIDEMAND}). Then there are $\epsilon_0>0$, $\lambda > 0$ and $T>0$ such that for every $\epsilon \in (0,\epsilon_0]$, $\lambda \in (0,\lambda_0]$ and pair of initial data for (\ref{E:HEQN}) $(H_0,\partial_\tau H_0)$ satisfying $\mathcal{S}^N(H_0,\partial_\tau H_0) \leq \epsilon$ and $\| H(0) \|^2_0 \leq \lambda$, there exists a unique solution $(H(\tau),H_\tau(\tau)):[0,1] \rightarrow \mathbb{R} \times \mathbb{R}$ to (\ref{E:HEQN})-(\ref{E:THETAICGAMMALEQ5OVER3}) for all $\tau \in [0,T]$. The solution has the property $\mathcal{S}^N(H,H_\tau) \lesssim \epsilon$ for each $\tau \in [0,T]$. Furthermore, the map $[0,T]\ni\tau\mapsto\mathcal{S}^N(\tau)\in\mathbb R_+$ is continuous.
\end{manualtheorem}
\begin{proof}[Sketch of proof] The proof follows by adapting the argument in \cite{JaMa2009,doi:10.1002/cpa.21517}. Notably, \cite{JaMa2009,doi:10.1002/cpa.21517} establishes the existence theory based on a suitable approximate scheme and a priori bounds. We will design the approximate scheme for $\mathfrak{H}:=D_r H$ and $H$ from $\mathfrak{H}$ from which we can apply the result of \cite{JaMa2009,doi:10.1002/cpa.21517}. The $j$th approximations $(\mathfrak{H}_j,\partial_\tau \mathfrak{H}_j)$ and $(H_j,\partial_\tau H_j)$ are constructed as follows. The initial data $(H_0,\partial_\tau H_0)$ such that $\mathcal{S}^N(H_0,\partial_\tau H_0) \leq \epsilon$ is used for the first approximation $j=1$, that is, we let $(\mathfrak{H}_1,\partial_\tau \mathfrak{H}_1)=(D_r H_0, D_r \partial_\tau H_0 )$ and $(H_1,\partial_\tau H_1)=(H_0,\partial_\tau H_0)$. Then we obtain the approximate $(j+1)^{th}$ solutions by induction: for $j \geq 1$, we let $(\mathfrak{H}_{j+1},\partial_\tau \mathfrak{H}_{j+1})$ solve the linear PDE
\begin{align}
&a^{3 \gamma-3} \partial_{\tau}^2 \mathfrak{H}_{j+1} + a^{3 \gamma-4} a_{\tau} \partial_\tau \mathfrak{H}_{j+1} - \gamma \frac{\xi^4}{\J^{\gamma+1}} \mathcal{L}_1 \mathfrak{H}_{j+1} - \mathfrak{H}_{j+1} \notag \\
&= \mathcal{D}_1 \left[\frac{(H_j)^2}{r}\right] - \mathcal{D}_1 \left[ r \mathcal{R}_1[\tfrac{H_j}{r}] + r \mathcal{R}_2[\tfrac{H_j}{r}] +r \mathcal{R}_3[\tfrac{H_j}{r}] \right] + \gamma C_1[H_j] + \gamma \partial_r \left( \frac{\xi^4}{\J^{\gamma+1}} \right) L_0 H_j, \label{E:LWPLINEAR}
\end{align}
with the initial data $(\mathfrak{H}_{j+1},\partial_\tau \mathfrak{H}_{j+1})|_{\tau=0}=(D_r H_0, D_r \partial_\tau H_0 ).$ Equation (\ref{E:LWPLINEAR}) mimics (\ref{E:PREESTIMATEHAPPLYDI}) for $i=1$. The right hand side (\ref{E:PREESTIMATEHAPPLYDI}) is evaluated using $H_j$ which is indicated by the subscript $j$. The bound $\mathcal{S}^N (H_j, \partial_\tau H_j) < \infty$ depends only on $\epsilon$ and hence we can apply the duality argument in \cite{JaMa2009,doi:10.1002/cpa.21517} to obtain the existence of $(\mathfrak{H}_{j+1},\partial_\tau \mathfrak{H}_{j+1})$. We then define $H_{j+1}$ by
\begin{equation}
H_{j+1}=\frac{1}{r^2} \int_0^r \mathfrak{H}_{j+1} (r')^2 \, dr'
\end{equation}
and have that $\mathcal{S}^N (H_{j+1},\partial_\tau H_{j+1})<\infty$ from a priori estimates with this bound depending only on $\epsilon$. Finally, as $j \rightarrow \infty$, we extract a subsequence and obtain the limit $(H,\partial_\tau H)$ of $(H_j,\partial_\tau H_j)$ that is a solution to (\ref{E:HEQN}) on $[0,T]$ for some $T=T(\epsilon)>0$ with $\mathcal{S}^N(H,\partial_\tau H)\lesssim \epsilon$.
\end{proof}

\section{Coercivity Estimates}\label{A:COERCIVITY}
We give a useful result which will let us handle time weights with negative powers which are present in our equation structure when $\gamma > \frac53$.
\begin{lemma}[Coercivity Estimate]\label{L:COERCIVITY}
Let $(H, H_\tau):[0,1] \rightarrow \mathbb R \times \mathbb R$ be a unique local solution to (\ref{E:HEQN})-(\ref{E:THETAICGAMMALEQ5OVER3}) on $[0,T]$ for $T>0$ fixed with $\|H(0)\|_0^2 < \infty$ and assume $(H, H_\tau)$ satisfies the a priori assumptions (\ref{E:SNAPRIORI})-(\ref{E:DOUBLEPARTIALRTHETAAPRIORI}). Fix $N \geq 8$. Let $k \geq N$ in (\ref{E:PHIDEMAND}). Fix $0 \leq i \leq N-1$. Then for all $\tau \in [0,T]$, we have
\begin{equation}
\| \D_i H \|^2_i \lesssim \sup_{0 \leq \tau' \leq \tau} \{ a^2 \| \D_i H_\tau \|^2_i \} + \| \D_i H(0) \|^2_i.
\end{equation}
\end{lemma} 
\begin{proof}
Applying the fundamental theorem of calculus and the time integrability of $a^{-1}$
\begin{align}\label{E:COERCIVITYPROOF1A}
\D_i H = \int_0^\tau \D_i H d \tau' + \D_i H (0) &= \int_0^\tau a^{-1} a \D_i H_\tau d \tau' + \D_i H (0) \notag \\
& \lesssim \sup_{0 \leq \tau \leq \tau'} \{ a \D_i H_\tau \} + \D_i H (0).
\end{align}
Therefore using Cauchy's inequality ($xy \lesssim x^2 + y^2,$ $x,y \in \mathbb{R}$) 
\begin{equation}\label{E:COERCIVITYPROOF1B}
\| \D_i H \|^2_i \lesssim \sup_{0 \leq \tau \leq \tau'} \{ a^2 \|  \D_i H_\tau \|^2_i \} + \| \D_i H (0)\|^2_i.
\end{equation}
\end{proof}

\end{document}